\documentclass[11pt,letterpaper]{amsart}
\usepackage[foot]{amsaddr}

\usepackage{mathtools}
\usepackage[T1]{fontenc}
\usepackage[latin9]{inputenc}
\usepackage{geometry}
\usepackage{wrapfig}
\usepackage{multicol}
\usepackage{graphicx}
\usepackage{soul}
\usepackage{xcolor}
\usepackage{amssymb}
\usepackage{placeins}
\usepackage{cite}
\usepackage{caption}
\usepackage{enumerate}
\usepackage{afterpage}
\usepackage{enumitem}
\usepackage{bmpsize}
\usepackage{hyperref}
\usepackage{tabu}
\usepackage{enumitem}
\numberwithin{equation}{section}
\usepackage{stmaryrd}
\usepackage{tikz}
\usetikzlibrary{matrix,graphs,arrows,positioning,calc,decorations.markings,shapes.symbols}
\makeatletter
\renewcommand{\email}[2][]{%
  \ifx\emails\@empty\relax\else{\g@addto@macro\emails{,\space}}\fi%
  \@ifnotempty{#1}{\g@addto@macro\emails{\textrm{(#1)}\space}}%
  \g@addto@macro\emails{#2}%
}
\makeatother

\newtheorem{theorem}{Theorem}[section]
\newtheorem{lemma}[theorem]{Lemma}
\newtheorem{proposition}[theorem]{Proposition}

{ \theoremstyle{definition}
\newtheorem{definition}[theorem]{Definition}}
{ \theoremstyle{remark}
\newtheorem{remark}[theorem]{Remark}}

\usepackage{caption}
\usepackage{bbm}

\newcommand{\f}{\frac}

\title{GUE corners process in boundary-weighed six-vertex models}

\author[E. Dimitrov]{Evgeni Dimitrov}
\author[M. Rychnovsky]{Mark Rychnovsky}
\email[E. Dimitrov]{edimitro@math.columbia.edu}
\email[M. Rychnovsky]{mrychnov@gmail.com}

\begin{document}

\maketitle

\begin{abstract}
We consider a class of probability distributions on the six-vertex model, which originate from the higher spin vertex models in \cite{borodin2018higher} and have previously been investigated in \cite{dimitrov2016six}. For these random six-vertex models we show that the asymptotic behavior near their base is asymptotically described by the GUE-corners process. 
\end{abstract}

\tableofcontents

%
\section{Introduction}\label{Section1} In Section \ref{Section1.1} we describe the general structure of the measures we consider in this paper and discuss two different examples that have been previously studied. In Section \ref{Section1.2} we formulate our model precisely, explain why we are interested in it and present the main result we prove about it.

%
\subsection{Preface}\label{Section1.1}

The six-vertex model is a well studied exactly solvable model in statistical mechanics. Linus Pauling introduced the model in 1935 to describe the residual entropy of ice crystals. In addition to its original purpose of describing ice, the six-vertex model has been useful in understanding other physical phenomena such as phase transitions in magnetism \cite{baxter2016exactly, lieb1980two}. 

In the present paper we consider a family of six-vertex models on the half-infinite strip $D_n = \mathbb{Z}_{\geq 0} \times \{1, \dots,n\}$ where $n \in \mathbb{N}$. Specifically, the state space of the models is the set $\mathcal{P}_n$ consisting of all collections of $n$ up-right paths, with nearest neighbor steps in $D_n$ with the paths starting from the points $\{(0,i): 1 \leq i \leq n \}$ and exiting the top boundary. We add the additional condition, that no two paths can share a horizontal or vertical edge, see Figure \ref{S1_1}. 
\begin{figure}[h]
\begin{minipage}{.45\textwidth}
\centering
\includegraphics[scale=0.4]{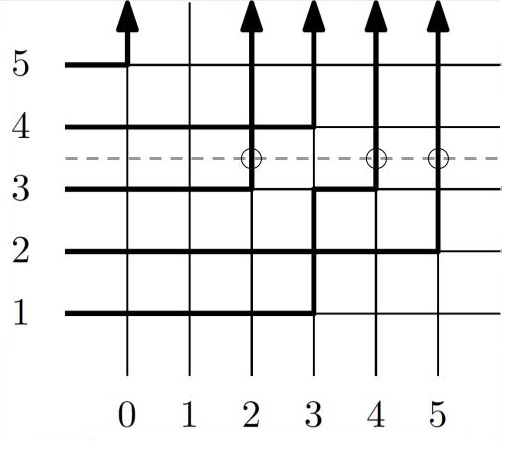}
 \captionsetup{width= \linewidth}
\caption{An example of a path collection $\pi$ in $\mathcal{P}_5$. Here $\lambda^3_1(\pi) = 5$, $\lambda^3_2(\pi) = 4$, $\lambda^3_3(\pi) = 2$}\label{S1_1}
\end{minipage}
\begin{minipage}{.45\textwidth}
\centering
\vspace{2.6mm}
\includegraphics[scale=1.3]{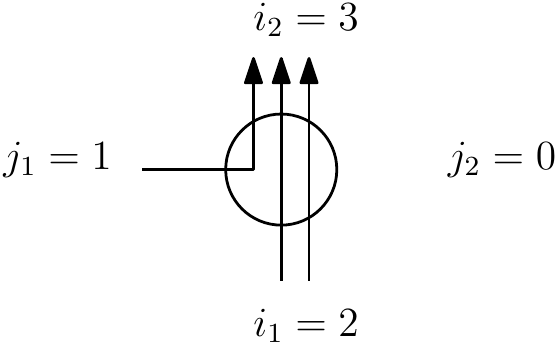}
 \captionsetup{width=0.9\linewidth}
\caption{An example of a vertex  of type $(i_1, j_1; i_2, j_2)=(2, 1; 3, 0)$}
\label{S1_2}
\end{minipage}
\end{figure}

In the next few paragraphs we explain the types of probability measures we put on $\mathcal{P}_n$ (they are given in equation (\ref{GeneralMeasure}) below), but to accomplish this we need a bit of notation. A \emph{signature} of length $n$ is a nonincreasing sequence $\lambda=(\lambda_1 \geq \lambda_2 \geq \dots \geq \lambda_n)$ with $\lambda_i \in \mathbb{Z}$. We use $\mathsf{Sign}_n$ to denote the set of all signatures of length $n$, and use $\mathsf{Sign}_n^{+}$ for the set of such signatures with $\lambda_n \geq 0$. To each collection of $n$ up-right paths $\pi \in \mathcal{P}_n$ one can identify a sequence of signatures $\lambda^i(\pi) \in \mathsf{Sign}^+_i$ for $i = 1,\dots,n$, where $(\lambda^i_1(\pi),\lambda^i_2(\pi),\dots,\lambda^i_i(\pi))$ are the ordered $x$-coordinates at which the paths in $\pi$ intersect the horizontal line $y=i+1/2$, see Figure \ref{S1_1}.  

Given an up-right path collection $\pi \in \mathcal{P}_n$, each vertex is given a \emph{vertex type} based on four numbers $(i_1,j_1;i_2,j_2)$, where $i_1$ and $j_1$ denote the number paths entering the vertex vertically and horizontally respectively, while $i_2$ and $j_2$ denote the number of paths leaving the vertex vertically and horizontally respectively, see Figure \ref{S1_2}. For complex parameters $s$ and $u$ we define the following vertex weights
\vspace{-3mm}
\begin{equation}\label{eq:weightssix}
\begin{split}
&w_1=w(0,0;0,0)=1, \hspace{38.5mm}  w_2=w(1,1;1,1) = \frac{u - s^{-1}}{1-su} \\
&w_2 =w(1,0;1,0) = \frac{1 - s^{-1}u}{1-su}, \hspace{25mm}  w_4=w(0,1;0,1) = \frac{u - s}{1-su} ,\\
& w_5=w(1,0; 0,1) = \frac{(1-s^2)u}{1-su}, \hspace{24.5mm}  w_6=w(0,1;1,0) = \frac{1 - s^{-2}}{1-su}.
\end{split}
\end{equation}
This nonintuitive parametrization of weights by $s$ and $u$ comes from \cite{borodin2018higher}, where it is important in defining a higher spin generalization of the six-vertex model. Later in (\ref{eq:weights}) we present the higher spin vertex weights, and one obtains the weights in (\ref{eq:weightssix}) by setting $q = s^{-2}$ in (\ref{eq:weights}). 

For $\pi \in \mathcal{P}_n$ we let $\pi(i,j)$ denote the vertex type of the vertex at position $(i,j)$ in the path collection $\pi$. Given complex numbers $s$ and $u$, and a function $f: \mathsf{Sign}_n^+ \to \mathbb{C}$ we define the weight of a path collection $\pi \in \mathcal{P}_n$ by
\vspace{-2mm}
$$\mathcal{W}^f(\pi)=f(\lambda^n(\pi)) \prod_{i=1}^{\infty} \prod_{j=1}^n w(\pi(i,j)).$$
All but finitely many $\pi(i,j)$ are equal to $(0,0;0,0)$ and have weight $1$ by (\ref{eq:weightssix}), so the product is well defined. If one chooses $u$ and $s$ in $\mathbb{C}$ and the function $f$ so that the weights $\mathcal{W}^f(\pi)$ are nonnegative, not all zero and summable then one can use the weights $W^f(\pi)$ to define a probability measure on $\mathcal{P}_n$ through
\begin{equation}\label{GeneralMeasure}
\mathbb{P}^f(\pi) = (Z^f)^{-1} \cdot \mathcal{W}^f(\pi), \mbox{ where }Z^f:=\sum_{\pi \in \mathcal{P}_n} \mathcal{W}^f(\pi).
\end{equation}
Equation (\ref{GeneralMeasure}) gives the general form of the measures we study in our paper. In plain words $\mathbb{P}^f$ is the usual six-vertex measure except that the path collections $\pi$ are reweighed based on their top boundary, namely $\lambda^n(\pi)$, through the boundary weight function $f$.
\begin{remark} When we go to our main results we will take $u > s > 1$ above. In the usual weight parametrization of the six-vertex model we have that 
\begin{equation*}
\begin{split}
&a_1 =1, \hspace{3.5mm}  a_2 =  \frac{u - s^{-1}}{su - 1}, \hspace{3.5mm} b_1 =  \frac{1 - s^{-1}u}{1-su}, \hspace{3.5mm}b_2 =  \frac{u - s}{su - 1}, \hspace{3.5mm} c_1 = \frac{(1-s^2)u}{1-su}, \mbox{ and }c_2 = \frac{1 - s^{-2}}{su-1}.
\end{split}
\end{equation*}
We mention that the latter weights are the absolute values of those in (\ref{eq:weightssix}), where ultimately the sign difference will be absorbed in the boundary weight function $f$ of the model so that $\mathcal{W}^f(\pi) \geq 0$ for all $\pi \in \mathcal{P}_n$. Associated with the six weights is an {\em anisotropy parameter} $\Delta$, given by
\begin{equation}\label{DefDelta}
\Delta(a_1, a_2, b_1, b_2, c_1, c_2) = \frac{a_1 a_2 + b_1b_2 - c_1 c_2}{2 \sqrt{a_1 a_2 b_1 b_2}},
\end{equation}
which is believed to be directly related with the qualitative and quantitative properties of the model, see \cite{Res10}. The choice of weights as in (\ref{eq:weightssix}) with $u > s > 1$ corresponds to $\Delta > 1$, which is known as the {\em ferroelectric phase} of the six-vertex model. 
\end{remark}

There are many different choices of parameters and functions $f$ that lead to meaningful measures in (\ref{GeneralMeasure}). For example, if $f(\lambda) = 0$ unless $\lambda_{n-i+1} = i-1$ for $i = 1, \dots, n$ the measure in $\mathbb{P}^f$ becomes the six-vertex model with {\em domain wall boundary condition} (DWBC), \cite{Kor82}. Another special case of the measures in (\ref{GeneralMeasure}) includes the case when $u > s > 1$ and 
\begin{equation}\label{EqStochBound}
f(\lambda) =\mathsf{G}_{\lambda}^c(\rho):= (-1)^n \cdot \mathbbm{1}_{m_0=0} \prod_{i=1}^{\infty}\mathbbm{1}_{m_i \leq 1} \prod_{j=1}^n (-s)^{\lambda_j},
\end{equation}
where $\lambda=0^{m_0} 1^{m_1} 2^{m_2}\dots $. In the latter notation $m_i$ is the number of times $i$ appears in the list $(\lambda_1, \dots, \lambda_n)$ and $\mathbbm{1}_E$ is the indicator function of the set $E$. With this choice of parameters and function $f$, the measure $\mathbb{P}^f$ becomes what is known as stochastic six-vertex model, see e.g. \cite{GwaSpohn92}, \cite{borodin2016stochastic}, with parameters
$$b_1 = \frac{1 - s^{-1} u}{1 - su}, \hspace{5mm} b_2 = \frac{s^2 - s u}{1 - su}.$$
For a quick proof of the latter statement we refer the reader to \cite[Section 6.5]{borodin2018higher}.

Different choices of the boundary weight function $f$ lead to qualitatively different behavior of the measures $\mathbb{P}^f$ in (\ref{GeneralMeasure}). We illustrate this point by comparing the DWBC and the stochastic six-vertex model we just introduced. In order to begin understanding the qualitative differences between these two models we need to discuss the {\em pure states} (or the ergodic, translation-invariant Gibbs measures) of the six-vertex model. For this we follow \cite{Amol2020}, see also \cite[Section 1.2.1]{CGST}.

A prediction in \cite{Bukman}, which has been very recently partially verified in \cite{Amol2020}, states that the pure states $\mu$ of the ferroelectric six-vertex model are parametrized by a {\em slope} $(s,t) \in [0,1]^2$, where $s$ and $t$ denote the probabilities that a given vertical and horizontal edge is occupied under $\mu$. For a certain open lens-shaped set $\mathfrak{H} \subset [0,1]^2$ one has the following characterization of pure states for the ferroelectric six-vertex model (here $\overline{\mathfrak{H}} = \mathfrak{H} \cup \partial \mathfrak{H}$):
\begin{enumerate}
\item {\em Nonexistence}: If $(s,t) \in \mathfrak{H}$, then there are no pure states $\mu_{s,t}$ of slope $(s,t)$.
\item {\em KPZ States}: If $(s,t) \in \partial \mathfrak{H}$, then $\mu_{s,t}$ should exhibit Kardar-Parisi-Zhang (KPZ) behavior.
\item {\em Liquid States}: If $(s,t) \in (0,1)^2 \setminus \overline{\mathfrak{H}}$, then $\mu_{s,t}$ should exhibit Gaussian free field (GFF) behavior. 
\item {\em Frozen States}: If $(s,t)$ is on the boundary of $[0,1]^2$, then $\mu_{s,t}$ should be frozen.  
\end{enumerate} 
From the above conjectural classification, \cite{Amol2020} established the nonexistence statement (1) and proved the existence and uniqueness of KPZ states (2) for all $(s,t) \in \partial \mathfrak{H}$. It is worth mentioning that the above classification sharply contrasts the one for dimer models. Specifically, the pure states in dimer models were classified in \cite{Sheffield1} and \cite{Sheffield2} and they come in three types. The first is {\em frozen}, where the associated height function is basically deterministic; the second is {\em gaseous}, where the variance of the height function is bounded but non-zero; the third is {\em liquid}, where the hight function fluctuations diverge logarithmically in the lattice size. In particular, for dimer models there are no Nonexistence or KPZ pure states. \\

Going back to our previous discussion, the stochastic six-vertex model considered in \cite{borodin2016stochastic}, which corresponds to $f$ as in (\ref{EqStochBound}), was shown to asymptotically have a phase diagram that consists of two frozen regions, i.e. regions where the local behavior of the model is described by Frozen States, and a non-frozen region, where one observes solely KPZ States, see Figure \ref{S1_4}. More specifically, in \cite{borodin2016stochastic} it was shown that the one-point marginals of the height function $h(x,y)$, which at a location $(x,y)$ counts the number of horizontal edges crossed by the vertical segment connecting $(x,0)$ and $(x,y)$ in the non-frozen region $III$ of Figure \ref{S1_4} are asymptotically governed by the GUE Tracy-Widom distribution \cite{TWPaper}. This type of behavior is characteristic of models in the KPZ universality class (for more background on this class we refer to the excellent survey \cite{CorwinKPZ}). For the DWBC six-vertex model a very different phase diagram is expected, although we emphasize that it has not been established rigorously. Specifically, for the DWBC it is expected that as $n$ becomes large the model again develops macroscopic frozen regions that are separated by a non-frozen region where one observes solely Liquid States. The only instance where this has been rigorously established is when $\Delta = 0$, which is the free fermion point of the model, see \cite{Kenyon1}, \cite{Kenyon2}. When $\Delta = 0$ the six-vertex model falls into the framework of the dimer models, which is what enables its precise mathematical analysis. We mention; however, that there are non-rigorous physics works and numerical simulations that indicate that for all DWBC six-vertex models one observes solely Liquid States in the non-frozen region, and by analogy with the dimer models the fluctuations of those are no longer KPZ, but rather governed by a suitable pullback of the Gaussian free field, \cite{Granet}. 
\begin{figure}[h]
\includegraphics[scale=0.4]{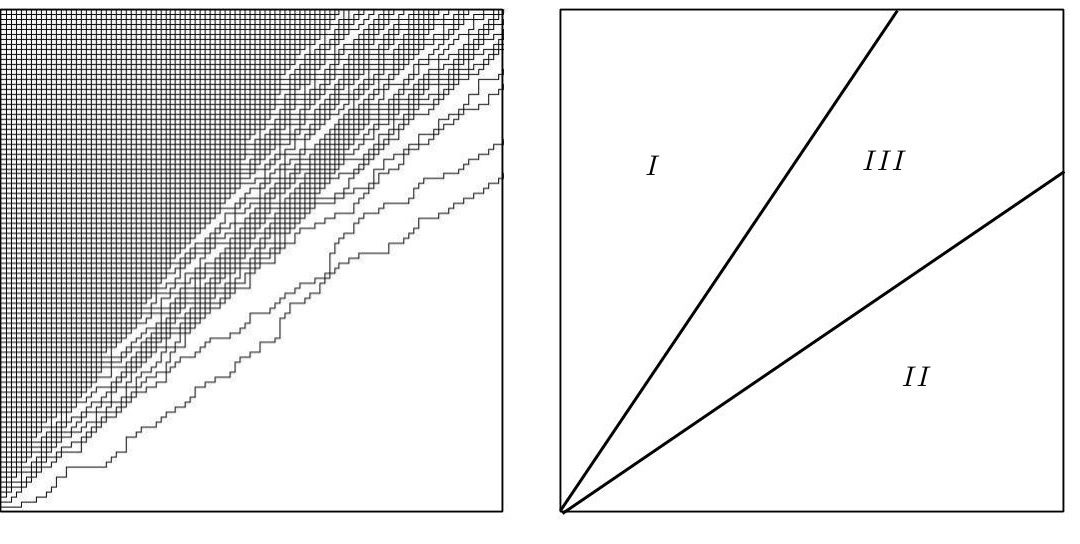}
 \captionsetup{width=0.9\linewidth}
\caption{The left picture represent a sample of $\mathbb{P}^f$ with $f$ as in (\ref{EqStochBound}) for the parameters $n = 100$, $u = 2$, $s^{-2} = 0.5$. The picture on the right side depicts the phase diagrams for these measures when $n$ is large. The regions $I$ and $II$ correspond to Frozen States and region $III$ corresponds to KPZ States
}
\label{S1_4}
\end{figure}

The above paragraphs explain that by picking different boundary weight functions $f$ we can obtain qualitatively different phase diagrams for our six-vertex model. In the present paper we consider a very special class of boundary functions $f$. This class will be described in the next section, where the definition of the models, some of their structural properties and main result we prove for them are presented. In the remainder of this section we explain the very high level motivations that have guided our choice of $f$. 

First of all, our discussion above indicates that for the stochastic six-vertex model of \cite{borodin2016stochastic} the non-frozen region consists entirely of KPZ States, while for the DWBC (at least conjecturally) it consists solely of Liquid States (or states with Gaussian statistics). A natural question is whether we can find a boundary weight function $f$ for which both types of pure states co-exist in the non-frozen region of the model. A second point is that, for general functions $f$, the asymptotic analysis for $\mathbb{P}^f$ is prohibitively complicated -- indeed even for the DWBC the phase diagram is largely conjectural, and so one is inclined to consider special boundary weight functions $f$ for which the analysis of the model is tractable. Our choice of $f$ is motivated by our desire that the resulting model satisfies these two properties.
%
\subsection{Model and results}\label{Section1.2}
In the present paper we study a special case of (\ref{GeneralMeasure}) when the boundary weight function $f$ is given by
\begin{equation}\label{EqOurBound}
f(\lambda) = \sum_{\mu \in \mathsf{Sign}_n^+} \mathsf{G}_{\mu}^c(\rho) \mathsf{G}^c_{\lambda/\mu}(v,\dots,v).
\end{equation}
In (\ref{EqOurBound}) the function $\mathsf{G}_{\mu}^c(\rho)$ is as in (\ref{EqStochBound}) and the functions $\mathsf{G}^c_{\lambda/\mu}$ are a remarkable class of symmetric rational functions, which were introduced in \cite{Bor14}. In the present paper one can find the definition of $\mathsf{G}^c_{\lambda/\mu}$ in Definition \ref{def:G}, and these functions depend on $M$ complex variables $v_1, \dots, v_M$ that have all been set to the same complex number $v$ in (\ref{EqOurBound}). We mention that $\mathsf{G}^c_{\lambda/\mu}$ are one-parameter generalizations of the classical (skew) Hall-Littlewood symmetric functions \cite{Mac} and carry the name of (skew) {\em spin Hall-Littlewood symmetric functions}, see \cite{BufPet}. 

One can check that if $v^{-1} > u > s > 1$ then the measure $\mathbb{P}^f$ from (\ref{GeneralMeasure}) with $f$ as in (\ref{EqOurBound}) is a well-defined probability measure, see Section \ref{Section2.2}. We will denote this measure by $\mathbb{P}^{N,M}_{u, v}$. \\

Even though the choice of $f$ in (\ref{EqOurBound}) seems complicated we emphasize that the resulting measure $\mathbb{P}^f$ enjoys many remarkable properties and its asymptotic structure appears to be rich and interesting. We elaborate on these points in the next few paragraphs, summarizing some results from \cite{dimitrov2016six} where this model was studied in detail.

First of all, the choice of $f$ as in (\ref{EqOurBound}) makes the model integrable and the distribution $\mathbb{P}^f$ analogous to the {\em ascending Macdonald processes} of \cite{borodin2014macdonald}. What plays the role of the (skew) Macdonald symmetric functions $P_{\lambda / \mu}$ and their duals $Q_{\lambda / \mu}$ is a class of symmetric rational functions $\mathsf{F}_{\lambda/ \mu}$ and their duals $\mathsf{G}^c_{\lambda/\mu}$ that were mentioned above. The functions $\mathsf{F}_{\lambda/\mu}, \mathsf{G}^c_{\lambda/\mu}$ enjoy many of the same properties as the Macdonald symmetric functions, including branching rules, orthogonality relations, (skew) Cauchy identities and so on. One consequence of the integrability of the model that can be appreciated by readers unfamiliar with symmetric function theory is that the partition function $Z^f$ for our choice of $f$ in (\ref{EqOurBound}) takes the following extremely simple product form
$$Z^f=(s^{-2};s^{-2})_n \left( \frac{ 1- s^{-1} u}{1-s u} \right)^n \left( \frac{1-s^{-2} uv}{1-uv} \right)^{nM}, \mbox{ where } (a;q)_m = (1-a)(1-aq) \cdots (1-aq^{m-1}).$$
The latter formula for the partition function is recalled in Section \ref{Section2.2} in the paper.

Another consequence of the integrability of the model is the fact that it is self-consistent in the following sense. Suppose that we sample a path collection $\pi$ on $\mathcal{P}_n$ according to $\mathbb{P}^{n,M}_{u, v}$ and then project the path collection to the first $k$ rows where $1 \leq k \leq n$. The resulting path collection is now a random element in $\mathcal{P}_k$ and its distribution is precisely $\mathbb{P}^{k,M}_{u, v}$ -- we recall this in Lemma \ref{Projections}. This self-consistency of the measures $\mathbb{P}^{n,M}_{u, v}$ for $n \in \mathbb{N}$ allows us for example to define a measure on up-right paths on the whole of $\mathbb{Z}_{\geq 0}^2$ whose projection to the bottom $n$ rows has law $\mathbb{P}^{n,M}_{u, v}$. 

\begin{figure} [h]
\includegraphics[scale=0.35]{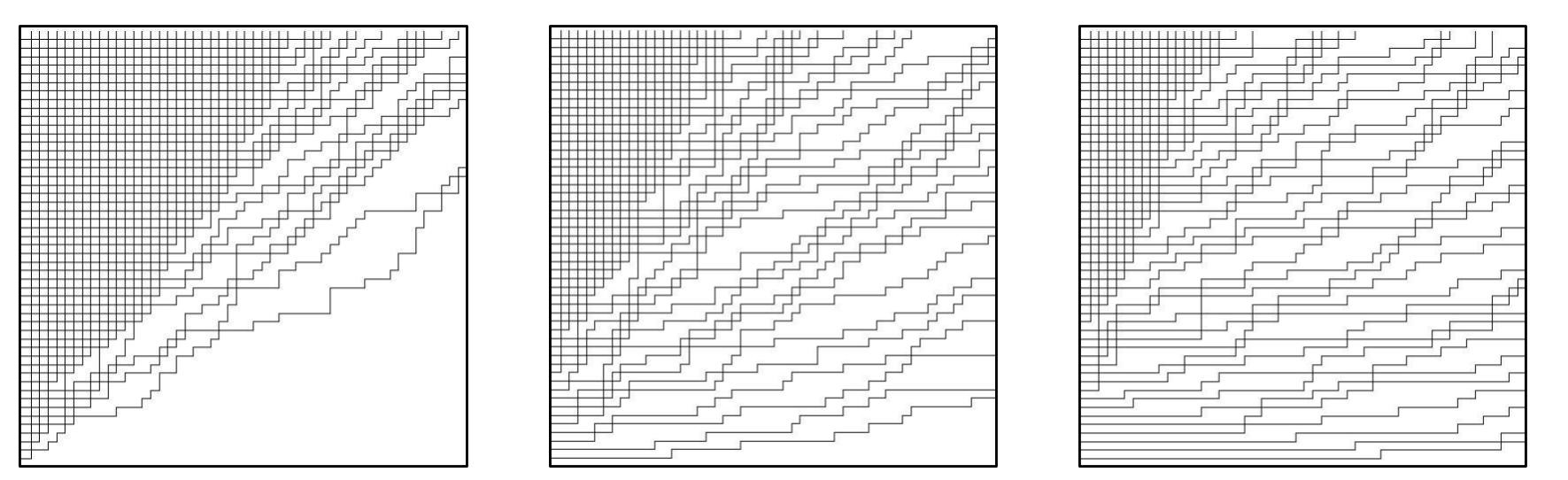}
\vspace{-6mm}
\caption{The pictures represent samples of the Markov chain $\{X_m\}_{m = 0}^\infty$ when $n = 50$ at times $m = 0$, $m = 50$ and $m = 100$. The parameters of the process are $s^{-2} = 0.5$, $u =2$ and all $v = 0.25$}
\label{S1_3}
\end{figure}
Yet another consequence of the integrability of the model is given by the fact that for fixed $n$ and varying $m \in \mathbb{Z}_{\geq 0}$ the measures $\mathbb{P}^{n,M}_{u, v}$ can be stochastically linked as we next explain. One can interpret the distribution $\mathbb{P}^{n,m}_{u, v}$ as the time $m$ distribution of a Markov chain $\{X_m\}_{m = 0}^\infty$ taking values in $\mathcal{P}_n$ for each $m$. This Markov chain is started from the stochastic six-vertex model at time zero, and its dynamics are governed by sequential update rules. For more details and a precise formulation we refer the reader to \cite[Section 6]{borodin2018higher} as well as \cite[Section 8]{EDArxiv} where an exact sampling algorithm of this process was developed by one of the authors. For a pictorial description of how the configurations $X_m$ evolve as time increases see  Figure \ref{S1_3}. This interpretation is similar to known interpretations of the Schur process and Macdonald process as fixed time distributions of certain Markov processes, see \cite{borodin2010schur, borodin2014macdonald}. \\

The above few paragraphs explained some of the structure and relationships between the measures $\mathbb{P}^{N,M}_{u,v}$ for varying $N, M \in \mathbb{N}$. These measures arise as degenerations of the higher-spin vertex models that were studied in \cite{borodin2018higher}, which is the origin of their integrability.  For the purposes of the present paper, the main consequence of the integrability of the model that is utilized is that one has suitable for asymptotic analysis formulas for the one-dimensional projections of $\mathbb{P}^{N,M}_{u,v}$. This is what makes the analysis of the model tractable, which as we recall from the end of Section \ref{Section1.1} is one of our desired properties. 

Our primary probabilistic interest in the measures $\mathbb{P}^{N,M}_{u,v}$ comes from the fact that as $N,M \rightarrow \infty$ the phase diagram of the model (at least conjecturally) exhibits all three types of pure states -- Frozen, Liquid and KPZ. The presence of all three types of pure states is the second high-level motivation behind our choice of $f$ as in (\ref{EqOurBound}) and we illustrate the phase diagram in Figure \ref{SS_1}.
\begin{figure} [h]
\includegraphics[scale=0.5]{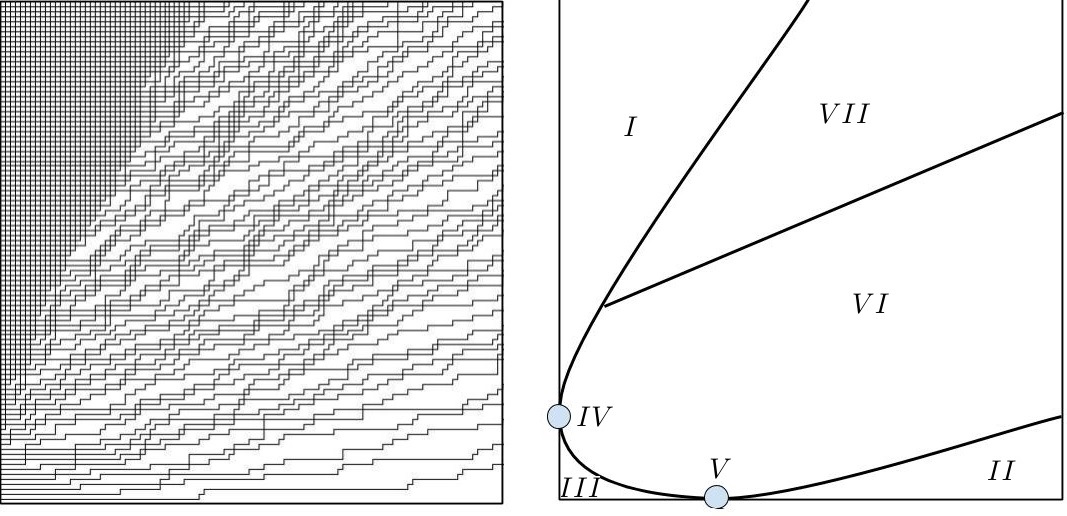}
\caption{The picture on the left represent a sample of $\mathbb{P}_{u,v}^{N,M}$ with $N = M = 100$, $u = 2$, $s^{-2} = 0.5$, $v = 0.25$. The picture on the right represents the (conjectural) phase diagram of the model as $N, M \rightarrow \infty$}
\label{SS_1}
\end{figure}
The phase diagram in Figure \ref{SS_1}, which corresponds to $\mathbb{P}_{u,v}^{N,M}$ when $N$ and $M$ are large, should be compared to the one in Figure \ref{S1_4}, which corresponds to the stochastic six-vertex model or equivalently to the measure $\mathbb{P}_{u,v}^{N,0}$ (recall that the measures $\mathbb{P}_{u,v}^{N,m}$ were stochastically linked through a Markov chain with time zero distribution given precisely by the stochastic six-vertex model). At least based on the simulations one observes that as the vertex model evolves in time $m = 0, \dots, M$ the frozen regions $I$ and $II$ from the stochastic six-vertex model in Figure \ref{S1_4} begin to deform and a new frozen region, denoted by $III$ in Figure \ref{SS_1} and consisting of vertices of type $(0,1;0,1)$, is formed near the origin. With this new frozen region two new points $IV$ and $V$ are formed. These are sometimes referred to as {\em turning points} and they arise where two different frozen regions meet each other. Furthermore, our prediction is that, under the Markovian dynamics evolving the six-vertex model, the KPZ cone (i.e. region $III$ in Figure \ref{S1_4}) that is present at time $m =0$ is translated away from the origin to region $VII$ and a new GFF region (denoted by $VI$ in Figure \ref{SS_1}) takes its place. We mention here that the exact nature of the Markovian dynamics is not important for this paper. The reason we mention it is to emphasize that the stochastic six-vertex model and the measures $\mathbb{P}_{u,v}^{N,M}$ we consider here are related to each other and the presence of the KPZ region $VII$ in $\mathbb{P}_{u,v}^{N,M}$ can be traced back to the presence of region $III$ in $\mathbb{P}_{u,v}^{N,0}$. If the same dynamics are run from a different initial configuration one may very well see a completely different phase diagram than the one in Figure \ref{SS_1}.  

As can be seen from Figure \ref{SS_1} the asymptotics of  $\mathbb{P}^{N,M}_{u,v}$ as $N,M \to \infty$ appear to be quite complex. A long term program, initiated in \cite{dimitrov2016six}, is to rigorously establish the phase diagram in Figure \ref{SS_1}. So far only the asymptotics near the point $IV$ have  been understood. Specifically, in \cite{dimitrov2016six} one of the authors showed that near $IV$ a certain configuration of empty edges converges to the GUE-corners process, we define the latter here. Recall that the Gaussian Unitary Ensemble (GUE) is a measure on Hermitian matrices $\{X_{ij}\}_{i,j=1}^k$ with density proportional to $e^{-Tr(X^2)/2}$ with respect to Lebesgue measure. For $1 \leq r \leq k$, let $\lambda_1^{r} \leq \lambda_2^{r} \leq \cdots \leq \lambda^{r}_{r}$ denote the ordered eigenvalues of the submatrix $\{X_{ij}\}_{i,j=1}^{r}$ of $X$. The joint law of the eigenvalues $\{\lambda^j_i\}_{1 \leq i \leq j \leq k}$ is called the \emph{GUE-corners process} of rank $k$ (or the GUE-minors process). The appearance of the GUE-corners process has been established in related contexts for random lozenge tilings in \cite{johansson2006eigenvalues, nordenstam2009interlaced, okounkov2006birth} and the uniform six-vertex model with domain-wall boundary conditions \cite{Gor14}. It is believed to be a universal scaling limit near points separating two different frozen regions such as the point $IV$.

The present paper, is a continuation of the program  initiated in \cite{dimitrov2016six} of establishing the phase diagram in Figure \ref{SS_1}. Specifically, in Figure \ref{SS_1} the point $V$ is another turning point, and in this paper, we show that the statistics of the model $\mathbb{P}^{N,M}_{u,v}$ near this point are also described by the GUE-corners process. Before we state our main result we give our choice of parameters and some notation.
\begin{definition}\label{DefConstants} We assume that $v, u,s \in (0, \infty)$ satisfy $v^{-1} > u > s > 1$. With this choice of parameters we define the constants
\begin{equation}\label{S1constants}
\begin{split}
&a=\frac{ v \left(u-s^{-1}\right) \left(s^{-1} u-1\right)}{(1-u v) (1- s^{-2} u v)}, \qquad b=\frac{ (s^2-1)}{(u-s)(1-su)} \\
& c=\frac{1}{2} \left(a \left(\frac{1}{(u-s)^2}-\frac{s^2}{(1-s u)^2}\right)-\frac{s^{-4} v^2}{(1-s^{-2} u v)^2}+\frac{v^2}{(1-u v)^2} \right), \qquad d=\frac{- \sqrt{2 c}}{b}.
\end{split}
\end{equation}
If $v^{-1} > u > s > 1$ one observes that
$$a>0, \qquad b<0, \qquad c>0, \qquad d>0.$$
See Lemma \ref{LemmaPar} in the main text for a verification of this fact.
\end{definition}
The main result of the paper is as follows.
\begin{theorem} \label{thm:main} Suppose that $u,v,s,a,d$ are as in in Definition \ref{S1constants} and $k \in \mathbb{N}$ is given. Suppose that $N(M)$ is a sequence of integers such that $N(M) \geq k$ for all $M$ and let $\mathbb{P}^{N,M}_{u,v}$ be the measure on collections of paths $\pi \in \mathcal{P}_N$ as earlier in the section. Define the random vector $Y(N,M;k)$ through
\begin{equation}\label{DefY}
Y_i^j(N,M;k) = \frac{\lambda^j_{j-i+1}(\pi) - aM}{d\sqrt{M}} \mbox{ for $1 \leq i \leq j \leq k$.}
\end{equation}
Then the sequence $Y(N,M;k)$ converges weakly to the GUE-corners process of rank $k$ as $M \to \infty$. 
\end{theorem}
\begin{remark} In (\ref{DefY}) we reverse the order of $\lambda^j_i$ because the usual convention for signatures $\lambda =(\lambda_1, \dots, \lambda_n)$ demands that $\lambda_i$ be sorted in decreasing order, while for the eigenvalues of a random matrix the usual convention is that they are sorted in increasing order.
\end{remark}

We mention here that while the asymptotic behaviors near $IV$ and $V$ are similar, the arguments used to establish the two results are quite different. The arguments in \cite{dimitrov2016six} rely on a remarkable class of difference operators, which can be used to extract averages of observables for $\mathbb{P}^{N,M}_{u,v}$ near the left boundary of the model. These observables become useless for accessing the asymptotic behavior of the base of the model and consequently our approach in this paper is completely different, and arguably more direct as we explain here. In the remainder of this section we give an outline of our approach to proving Theorem \ref{thm:main}. The discussion below will involve certain expressions that will be properly introduced in the main text, and which should be treated as black boxes for the purposes of the outline. 

Using the integrability of the model we obtain the formula
$$\mathbb{P}_{u,v}^{N,M}( \lambda_1^k(\pi) = \mu_1, \cdots, \lambda_1^k(\pi) = \mu_k) \propto \mathsf{F}_\mu([u]^k) \cdot f(\mu; [v]^M, \rho),$$
for any $\mu =(\mu_1, \dots, \mu_k) \in \mathsf{Sign}^+_k$. A generalization of this fact appears as Lemma \ref{Projections} in the main text. We then derive certain combinatorial estimates for $ \mathsf{F}_\mu([u]^k) $ and a contour integral formula for $f(\mu; [v]^M, \rho)$ in Section \ref{Section3}, which are both suitable for studying the $M \rightarrow \infty$ limit of these expressions (for the function $\mathsf{F}_\mu([u]^k) $ the dependence on $M$ is reflected in the scaling of the signature $\mu$). The limit of the contour integral formula for $f(\mu; [v]^M, \rho)$ is derived in Section \ref{Section5} using a steepest descent argument, while the combinatorial estimates for $\mathsf{F}_\mu([u]^k)$ prove sufficient for taking its limit. Combining our two asymptotic results for $ \mathsf{F}_\mu([u]^k)$ and $f(\mu; [v]^M, \rho)$ we can prove that the sequence of random vectors in $\mathbb{R}^k$, given by $Y^k(N,M) = \left(Y_1^k(N,M;k), \dots, Y_k^k(N,M;k) \right)$ with $Y(N,M;k)$ as in Theorem \ref{thm:main} converges to the measure of the ordered eigenvalues of a random GUE matrix $\mu_{\textrm{GUE}}^k(dx_1,...,dx_k)$, given by
\begin{equation*} 
\mu_{\textrm{GUE}}^k (dx_1,...,dx_k)= {\bf 1}\{ x_k > x_{k-1} > \cdots > x_1\} \left(\frac{1}{\sqrt{2 \pi}} \right)^k \cdot \frac{1}{\prod_{i=1}^{k-1} i!} \cdot  \prod_{1 \leq i < j \leq k} (x_i-x_j)^2 \prod_{i=1}^k e^{-\frac{x_i}{2}} dx_i.
\end{equation*}
The last statement appears as Proposition \ref{PropHerm} in the text.

The above paragraph explains how we show that the top row of $Y(N,M;k)$ converges to the top row of the GUE-corners process of rank $k$. To obtain the full convergence statement we combine our top row convergence statement with the general formalism, introduced in \cite{dimitrov2016six},  involving Gibbs measures on interlacing arrays. In more detail, the top-row convergence of $Y(N,M;k)$ and the interlacing conditions 
$$\lambda^{i+1}_1(\pi) \geq \lambda^i_1(\pi) \geq \lambda^{i+1}_2(\pi) \geq \lambda^i_2(\pi) \geq \cdots \geq \lambda^i_i(\pi) \geq \lambda^{i+1}_{i+1}(\pi),$$
for $i = 1,\dots, k-1$ are enough to conclude the tightness of the full vector $Y(N,M;k)$ as $M \rightarrow \infty$. For each $N,M$ the measures $\mathbb{P}_{u,v}^{N,M}$ satisfy what we call the {\em six-vertex Gibbs property} and in the $M \rightarrow \infty$ limit this property becomes what is known as the {\em continuous Gibbs property}, see Section \ref{Section4.2}. Combining the latter statements, one can conclude that any subsequential limit of $Y(N,M;k)$ as $M \rightarrow \infty$ has top row distribution $\mu_{\textrm{GUE}}^k(dx_1,...,dx_k)$ and satisfies the continuous Gibbs property, and these two characteristics are enough to identify this limit with the GUE-corners process of rank $k$. As the sequence $Y(N,M;k)$ is tight and all subsequential limits are the same and equal to the GUE-corners of rank $k$, we conclude the weak convergence of $Y(N,M;k)$. This argument is given in Section \ref{Section4.2}.
\subsection{Outline of the paper}
In Section \ref{Section2} we define and state some facts about the functions $\mathsf{F}_{\lambda/\mu}$ and $\mathsf{G}^c_{\lambda/\mu}$, we also define inhomogeneous versions of $\mathbb{P}^{N,M}_{u,v}$ called $\mathbb{P}_{{\bf u}, {\bf v}}$ and give formulas for their projections. In Section \ref{Section3} we derive certain combinatorial estimates for $\mathsf{F}_\mu([u]^k)$ and a contour integral formula for $f(\mu;{\bf v}, \rho)$, where the latter appear in our projection formulas for $\mathbb{P}_{{\bf u}, {\bf v}}$ from Section \ref{Section2}. In Section \ref{Section4} we prove Theorem \ref{thm:main} modulo Lemma \ref{BMAs} , which is proved in Section \ref{Section5}.
\subsection{Acknowledgements}
The authors would like to thank Ivan Corwin and Ioannis Karatzas for many helpful conversations. M. R. is partially supported by the NSF grant DMS:1664650. E.D. is partially supported by the Minerva Foundation Fellowship.

%
\section{Measures on up-right paths} \label{Section2}
In this section we provide some results about a certain class of measures $\mathbb{P}_{{\bf u}, {\bf v}}$ that are inhomogeneous analogues of the measures $\mathbb{P}^{N,M}_{u,v}$ from Section \ref{Section1.2}. For the most part, this section summarizes the results in \cite[Section 2]{dimitrov2016six}.

%
\subsection{Symmetric rational functions} \label{Section2.1}
In this section we introduce some necessary notation from \cite{borodin2018higher} and summarize some of the results from the same paper. A {\em signature} of length $N$ is a sequence $\lambda = (\lambda_1 \geq \lambda_2 \geq \cdots \geq \lambda_N)$ with $\lambda_i \in \mathbb{Z}$ for $i = 1, \dots, N$. A signature $\lambda$ will sometimes be represented by $0^{m_0(\lambda)}1^{m_1(\lambda)}2^{m_2(\lambda)} \dots$ where $m_i(\lambda):=|\{ j : \lambda_j=i\}|$ is the number of times $i$ appears in the list $(\lambda_1,\dots,\lambda_N)$. We denote by $\mathsf{Sign}_N$ the set of all signatures of length $N$ and by $\mathsf{Sign}^+_N$ the set of signatures of length $N$ with $\lambda_N \geq 0$. We also denote by $\mathsf{Sign}^+:= \sqcup_{N \geq 0} \mathsf{Sign}^+_N$ the set of all non-negative signatures. We recall for later use the $q$-Pochhammer symbol $(a;q)_n:=(1-a)(1-q a)\cdots(1-q^{n-1} a).$

In what follows, we define the {\em weight} of a finite collection of up-right paths in some region $D$ of $\mathbb{Z}^2$, which is equal to the product of the weights of all vertices that belong to the path collection. Throughout we will always assume that the weight of an empty vertex is $1$ and so alternatively the weight of a path configuration can be defined through the product of the weights of all vertices in $D$. Figure \ref{S2_1} gives examples of collections of up-right paths.

The path configuration at  a vertex is determined by four non-negative integers $(i_1,j_1;i_2,j_2)$, where $i_1$ (resp. $i_2$) is the number of incoming (resp. outgoing) vertical paths, and $j_1$ (resp. $j_2$) is the number of incoming (resp. outgoing) horizontal paths, see Figure \ref{S1_2}. If the path configuration of a vertex is $(i_1,j_1;i_2,j_2)$ we will also say that the vertex is of type $(i_1,j_1;i_2,j_2)$. Vertex weights will be given as a function of these four numbers. We require the number of paths entering and leaving each vertex to be the same, i.e. $i_1+j_1=i_2+j_2$, and we will constrain the horizontal number of paths by $j_1,j_2 \in \{0,1\}$ (the weight of any vertex that does not satisfy these two conditions is $0$). 

We consider two sets of vertex weights depending on parameters $s$ and $q$ (these are fixed throughout this section), and a \emph{spectral parameter} $u$. The first set of vertex weights is given by
\begin{equation}\label{eq:weights}
\begin{split}
&w_u(g,0;g,0) = \frac{1 - sq^gu}{1-su}, \hspace{20mm} w_u(g+1,0; g,1) = \frac{(1-s^2q^g)u}{1-su}\\
&w_u(g,1;g,1) = \frac{u - sq^g}{1-su}, \hspace{22mm} w_u(g,1;g+1,0) = \frac{1 - q^{g+1}}{1-su},
\end{split}
\end{equation}
where $g$ is any nonnegative integer and all other weights are assumed to be $0$. The second set of weights, called \emph{conjugated} weights, are defined by
\begin{equation}\label{eq:weightsconjugate}
\begin{split}
&w^c_u(g,0;g,0) = \frac{1 - sq^gu}{1-su}, \hspace{20mm} w^c_u(g+1,0; g,1) = \frac{(1-q^{g+1})u}{1-su}\\
&w^c_u(g,1;g,1) = \frac{u - sq^g}{1-su}, \hspace{22mm} w^c_u(g,1;g+1,0) = \frac{1 - s^2q^{g}}{1-su},
\end{split}
\end{equation}
where as before  $g \in \mathbb{Z}_{\geq 0}$ and all other weights are assumed to be $0$. For more background and motivation for this particular choice of weights we refer the reader to \cite[Section 2]{borodin2018higher}.\\

Let us fix $n \in \mathbb{N}$, $n$ indeterminates $u_1, \dots, u_n$ and the region $D_n = \mathbb{Z}_{\geq 0} \times \{ 1, \dots, n \}$. Let $\pi$ be a finite collection of up-right paths in $D_n$, which end in the top boundary, but are allowed to start from the left or bottom boundary of $D_n$. We denote the path configuration at vertex $(i,j) \in D_n$ by $\pi(i,j)$. Then the weight of $\pi$ with respect to the two sets of weights above is defined through
\begin{equation}\label{defWeights}
\mathcal{W}(\pi) = \prod_{i = 0}^\infty \prod_{j = 1}^n w_{u_j}(\pi(i,j)), \hspace{20mm} \mathcal{W}^c(\pi) = \prod_{i = 0}^\infty \prod_{j = 1}^n w_{u_j}^c(\pi(i,j)).
\end{equation}
Note that from (\ref{eq:weights}) and (\ref{eq:weightsconjugate}) we have that $w_u(0,0;0,0) =  1 = w_u^c(0,0;0,0)$ and so the above products are in fact finite. With the above notation we define the following partition functions.
\begin{definition} \label{def:G} Let $N ,n \in \mathbb{Z}_{\geq 0}$, $\lambda, \mu \in \mathsf{Sign}_N^+$ and $u_1, \dots, u_n \in \mathbb{C}$ be given. Let $\mathcal{P}^c_{\lambda/\mu}$ be the collection of up-right paths $\pi$ on $D_n$, which
\begin{itemize}
\item begin with $N$ vertical edges $(\mu_i,0) \to (\mu_i,1)$ for $i=1,\dots,N$ along the horizontal axis;
\item end with $N$ vertical edges $(\lambda_i,n) \to (\lambda_i,n+1)$ for $i=1,\dots,N$. 
\end{itemize}Then we define
\begin{equation}\label{GdefEq}
\mathsf{G}^c_{\lambda / \mu} (u_1, \dots, u_n) = \sum_{\pi \in \mathcal{P}_{\lambda/\mu}^c} \mathcal{W}^c(\pi).
\end{equation}
We also use the abbreviation $\mathsf{G}_{\lambda}^c$ for $\mathsf{G}_{\lambda/(0,0,\dots, 0)}^c.$ 
\end{definition}
For the second set of weights we have an analogous definition.
\begin{definition}\label{DefF}
Let $N, n \in \mathbb{Z}_{\geq 0}$, $\mu \in \mathsf{Sign}^+_N$, $\lambda \in \mathsf{Sign}_{N+n}^+$ and $u_1, \dots, u_n \in \mathbb{C}$ be given. Let $\mathcal{P}_{\lambda/\mu}$ be the collection of up-right paths $\pi$ on $D_n$, which 
\begin{itemize}
\item begin with edges $(\mu_i,0) \to (\mu_i,1)$ for $i=1,\dots,N$ along the bottom boundary of $D_n$ and with edges $(-1,y) \to (0,y)$ for $y=1,\dots,n$ along the left boundary;
\item end with $N+n$ vertical edges $(\lambda_i,k) \to (\lambda_i,n+1)$ for $i=1,\dots,N+n$. 
\end{itemize}
Then we define
\begin{equation}\label{FdefEq}
\mathsf{F}_{\lambda/\mu}(u_1,\dots,u_n) =\sum_{\pi \in \mathcal{P}_{\lambda/\mu}} \mathcal{W}(\pi).
\end{equation}
We also use the abbreviation $\mathsf{F}_\lambda = \mathsf{F}_{\lambda/ \varnothing}.$
\end{definition}
Path configurations that belong to $\mathcal{P}_{\lambda/ \mu}$ and $\mathcal{P}^c_{\lambda/ \mu}$ are depicted in Figure \ref{S2_1}. In the definitions above we define the weight of a collection of paths to be $1$ if it has no interior vertices. Also, the weight of an empty collection of paths is $0$. \\
\begin{figure}[h]
\includegraphics[scale=0.6]{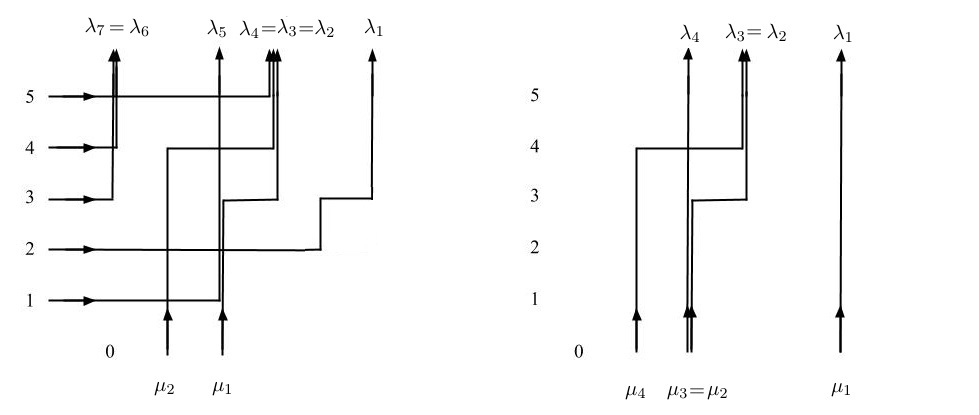}
 \captionsetup{width=0.9\linewidth}
\caption{Path collections belonging to $\mathcal{P}_{\lambda/\mu}$ (left) and $\mathcal{P}^c_{\lambda / \mu}$ (right).
}
\label{S2_1}
\end{figure}

Below we summarize some of the properties of the functions $\mathsf{G}^c_{\lambda/ \mu}$ and $\mathsf{F}_{\lambda/ \mu}$ in a sequence of propositions; we refer the reader to \cite[Section 4]{borodin2018higher} for the proofs. We mention here that the statements we write below for $\mathsf{G}_{\lambda/\mu}^c$ appear in \cite[Section 4]{borodin2018higher} for a slightly different but related function $\mathsf{G}_{\lambda/\mu}$. The function $\mathsf{G}_{\lambda/\mu}$ has the same definition as $\mathsf{G}_{\lambda/\mu}^c$ except that one uses the vertex weights (\ref{eq:weights}) rather than the conjugated weights (\ref{eq:weightsconjugate}). One directly checks that the two sets of weights are related through the equation
$$w_u^c(i_1, j_1;i_2, j_2) = \frac{(q;q)_{i_1} (s^2;q)_{i_2} }{(q;q)_{i_2}(s^2; q)_{i_1}} \cdot w_{u}(i_1, j_1;i_2, j_2),$$
which results in the relation
\begin{equation}\label{relation}
\mathsf{G}_{\lambda/\mu}^c=\frac{c(\lambda)}{c(\mu)} \cdot \mathsf{G}_{\lambda/\mu} , \mbox{ where }c(\lambda)=\prod_{k=0}^{\infty}\frac{(s^2;q)_{n_k}}{(q;q)_{n_k}} \mbox{ for $\lambda = 0^{n_0}1^{n_1}2^{n_2} \dots $}.
\end{equation}

\begin{proposition}\label{symmetricFun}\cite[Proposition 4.5]{borodin2018higher}
The functions $\mathsf{F}_{\lambda/\mu}(u_1,\dots ,u_n)$ and $\mathsf{G}^c_{\lambda/\mu}(u_1,\dots,u_n)$ defined above are rational symmetric functions in the variables $u_1,\dots,u_n$.
\end{proposition}

\begin{proposition}\label{Branching}\cite[Proposition 4.6]{borodin2018higher}
{\bf 1.} For any $N, n_1, n_2 \in \mathbb{Z}_{\geq 0}$, $\mu \in \mathsf{Sign}_N^+$ and $\lambda \in \mathsf{Sign}_{N + n_1 + n_2}^+$ one has
\begin{equation}\label{branchF}
\mathsf{F}_{\lambda / \mu}(u_1,\dots,u_{n_1+ n_2}) = \sum_{ \kappa \in \mathsf{Sign}_{N + n_1}^+ }\mathsf{F}_{\lambda / \kappa}(u_{n_1 + 1},\dots,u_{n_1 + n_2}) \mathsf{F}_{\kappa / \mu}(u_1,\dots,u_{n_1}).
\end{equation}
{\bf 2.} For any $N, n_1, n_2 \in \mathbb{Z}_{\geq 0}$ and $\lambda, \mu \in \mathsf{Sign}_{N}^+$, one has
\begin{equation}
\mathsf{G}^c_{\lambda/\mu}(u_1,\dots,u_{n_1 + n_2}) = \sum_{ \kappa \in \mathsf{Sign}_{N }^+ }\mathsf{G}^c_{\lambda / \kappa}(u_{n_1 + 1},\dots,u_{n_1 + n_2}) \mathsf{G}^c_{\kappa / \mu}(u_1,\dots,u_{n_1}).
\end{equation}
\end{proposition}
{\raggedleft The properties of the last proposition are known as {\em branching rules}.}

\begin{definition}
We say that two complex numbers $u,v \in \mathbb{C}$ are {\em admissible} with respect to the parameter $s$ if $\left| \frac{u - s}{1 - su} \cdot \frac{v -s }{1 - sv}  \right| < 1.$
\end{definition}

\begin{proposition}\label{PPieri}\cite[Corollary 4.10]{borodin2018higher} Let $u_1,\dots,u_N$ and $v_1,\dots,v_K$ be complex numbers such that $u_i,v_j$ are admissible for all $i = 1,\dots,N$ and $j = 1,\dots,K$. Then for any $\lambda, \nu \in \mathsf{Sign}^+$
\begin{equation}\label{CauchyId}
\begin{split}
&\sum_{\kappa \in \mathsf{Sign}^+} \mathsf{G}^c_{\kappa / \lambda} (v_1,\dots,v_K)\mathsf{F}_{\kappa/\nu}(u_1,\dots,u_N) =\\
& \prod_{i = 1}^N \prod_{j = 1}^K \frac{1 - qu_iv_j}{1 - u_iv_j} \sum_{ \mu \in \mathsf{Sign}^+}\mathsf{F}_{\lambda /\mu}(u_1,\dots,u_N) \mathsf{G}^c_{\nu / \mu} (v_1,\dots,v_K).
\end{split}
\end{equation}
\end{proposition}
\begin{remark}
Equation (\ref{CauchyId}) is called the {\em skew Cauchy identity} for the functions $\mathsf{F}_{\lambda / \mu}$ and $\mathsf{G}^c_{\lambda/ \mu}$ because of its similarity with the skew Cauchy identities for Schur, Hall-Littlewood, or Macdonald symmetric functions \cite{Mac}. The sum on the right-hand side (RHS) of (\ref{CauchyId}) has finitely many non-zero terms and is thus well-defined. The left-hand side (LHS) can potentially have infinitely many non-zero terms, but part of the statement of the proposition is that if the variables are admissible, then this sum is absolutely convergent and numerically equals the right side.
\end{remark}
A special case of (\ref{CauchyId}), when $\lambda = \varnothing$ and $\nu = (0,0,\dots,0)$ leads us to the {\em Cauchy identity}
\begin{equation}\label{CauchyS2}
\sum_{ \nu \in \mathsf{Sign}_N^+} \mathsf{F}_\nu(u_1,\dots,u_N) \mathsf{G}^c_\nu(v_1,\dots,v_K) = (q;q)_N \prod_{i = 1}^N \left( \frac{1}{1 - su_i} \prod_{j = 1}^K \frac{1 - qu_i v_j}{1 - u_iv_j}\right) .
\end{equation}

We end this section with the {\em symmetrization formulas} for $\mathsf{G}^c_\nu$ and $\mathsf{F}_\mu$ and also formulas for the functions when the variable set forms a geometric progression with parameter $q$.
\begin{proposition}\label{symmF}\cite[Theorem 4.14]{borodin2018higher}
{\bf 1.} For any $N \in \mathbb{Z}_{\geq 0}$, $\mu \in \mathsf{Sign}_N^+$ and $u_1,\dots,u_N \in \mathbb{C}$
\begin{equation}\label{symmFEq}
\mathsf{F}_{\mu}(u_1,\dots,u_{N}) = \frac{(1-q)^N}{\prod_{i = 1}^N(1 -su_i)}\sum_{\sigma \in S_N} \sigma \left( \prod_{1 \leq \alpha < \beta \leq N} \frac{u_{\alpha} - q u_\beta}{u_\alpha - u_\beta} \left( \frac{ u_i - s}{1 - su_i}\right)^{\mu_i}\right).
\end{equation}
{\bf 2.} Let $n \geq 0$ and $\mathsf{Sign}_n^+ \ni \nu = 0^{n_0}1^{n_1}2^{n_2}\cdots$. Then for any $N \geq n - n_0$ and $u_1,\dots,u_N \in \mathbb{C}$ we have
\begin{equation}\label{symmGEq}
\begin{split}
\mathsf{G}^c_{\nu}(u_1,\dots,u_{N}) = \frac{(1-q)^N(q;q)_n}{\prod_{i = 1}^N(1 - su_i)(q;q)_{N- n + n_0}(q;q)_{n_0}} \prod_{k = 1}^{\infty} \frac{(s^2;q)_{n_k}}{(q;q)_{n_k}} \times \\ 
 \sum_{\sigma \in S_N} 
\sigma \left( \prod_{1 \leq \alpha < \beta \leq N} \frac{u_{\alpha} - q u_\beta}{u_\alpha - u_\beta} \prod_{i = 1}^n \left( \frac{ u_i - s}{1 - su_i}\right)^{\nu_i} \prod_{i = 1}^{n- n_0}\frac{u_i}{u_i - s}\prod_{j = n-n_0+1}^N(1 -sq^{n_0}u_j)\right).
\end{split}
\end{equation}
In both equations above, $S_N$ denotes the permutation group on $\{1,\dots,N\}$ and an element $\sigma \in S_N$ acts on the expression by permuting the variable set to $u_{\sigma(1)},\dots,u_{\sigma(N)}$. If $N < n-n_0$, then $\mathsf{G}^c_{\nu}(u_1,\dots,u_{N})$ is equal to $0$.
\end{proposition}
\begin{remark} We mention here that the formulas in Proposition \ref{symmF} a priori make sense when $u_{\alpha} \neq u_{\beta}$ for $\alpha \neq \beta$ because the factors $u_\alpha - u_\beta$ appear in the denominator on the right sides of (\ref{symmFEq}) and (\ref{symmGEq}). However, the formulas can be extended continuously to all $(u_1, \dots, u_N)$ such that $u_i \not \in \{s, s^{-1} \}$ for all $i = 1, \dots, N$. One observes this after putting all summands under the same denominator and realizing that the numerator is a skew-symmetric polynomial in $(u_1, \dots, u_N)$, which is thus divisible by the Vandermonde determinant $\prod_{1 \leq \alpha < \beta \leq N}(u_\alpha - u_\beta)$. 
\end{remark}

\begin{proposition}\label{qGeom1}
{\bf 1.} For any $N \in \mathbb{Z}_{\geq 0}$, $\mu \in \mathsf{Sign}_N^+$ and $u \in \mathbb{C}$, one has
\begin{equation}
\mathsf{F}_{\mu}(u,qu,\dots,q^{N-1}u) = (q;q)_N\prod_{i = 1}^N \left ( \frac{1}{1 - sq^{i-1}u}  \left( \frac{ uq^{i-1} - s}{1 - sq^{i-1}u}\right)^{\mu_i}\right).
\end{equation}
{\bf 2.} Let $n \geq 0$ and $\mathsf{Sign}_n^+ \ni \nu = 0^{n_0}1^{n_1}2^{n_2}\cdots$. Then for any $N \geq n - n_0$ and $u \in \mathbb{C}$ we have
\begin{equation}
\begin{split}
\mathsf{G}^c_{\nu}(u,qu,\dots,q^{N-1}u) \hspace{-1mm} =  \hspace{-1mm} \prod_{k = 1}^\infty \frac{(s^2;q)_{n_k}}{(q;q)_{n_k}}\frac{(q;q)_N (su;q)_{N+n_0}(q;q)_n  \prod_{i = 1}^N \left( \frac{1}{1 - sq^{i-1}u} \left( \frac{q^{i-1}u - s}{1 - sq^{i-1}u}\right)^{\nu_j} \right)}{(q;q)_{N-n+n_0}(su;q)_n(q;q)_{n_0}(su^{-1};q^{-1})_{n-n_0}},
\end{split}
\end{equation}
where we agree that $\nu_j = 0$ if $j > n$. 
\end{proposition}

%
\subsection{The measure $\mathbb{P}_{{\bf u}, {\bf v}}$}\label{Section2.2}
In this section we briefly explain how to construct the measure $\mathbb{P}_{{\bf u}, {\bf v}}$ and summarize some of its basic properties. For a more detailed derivation of this measure we refer the reader to \cite[Sections 2.2 and 2.3]{dimitrov2016six}.

Let us briefly explain the main steps of the construction of $\mathbb{P}_{{\bf u, v}}$. We begin by considering the bigger space of all up-right paths in the half-infinite strip that share no horizontal piece but are allowed to share vertical pieces. For each such collection of paths we define its weight using the functions from Section \ref{Section2.1}. Afterwards we specialize $s = q^{-1/2}$ in those functions and perform a limit transition for some of the other parameters. This procedure has the effect of killing the weight of those path configurations that share a vertical piece. Consequently, we are left with weights that are non-zero only for six-vertex configurations, are absolutely summable and their sum has a product form. We check that each weight is non-negative, and define $\mathbb{P}_{{\bf u, v}}$ as the quotient of these weights with the partition function. We explain this in more detail below.\\

We fix positive integers $N, M$, $J$, and $K = M + J$, as well as real numbers $q \in (0,1)$ and $s> 1$. In addition, we suppose ${\bf u} = (u_1,\dots,u_N)$ and ${\bf w} = (w_1,\dots,w_K)$ are positive real numbers, such that $\max_{i,j} u_i w_j < 1$ and $u:=\min_i u_i > s$. One readily verifies that the latter conditions ensure that $u_i, w_j$ are admissible with respect to $s$ for $i = 1,\dots,N$ and $j = 1,\dots,K$.

Let us go back to the setup of Section \ref{Section1.2}. We let $\mathcal{P}'_N$ be the collection of $N$ up-right paths drawn in the sector $D_N = \mathbb{Z}_{\geq 0} \times \{1,\dots,N\}$ of the square lattice, with all paths starting from a left-to-right edge entering each of the points $\{ (0,m): 1 \leq m \leq N\}$ on the left boundary and all paths exiting from the top boundary of $D_N$. We still assume that no two paths share a horizontal piece, but sharing vertical pieces is allowed. As in Section \ref{Section1.2} we let $\mathcal{P}_N \subset \mathcal{P}_N'$ be those collections of paths that do not share vertical pieces. For $\pi \in \mathcal{P}'_N$ and $k = 1,\dots,N$ we let $\lambda^k(\pi) \in \mathsf{Sign}_k^+$ denote the ordered $x$-coordinates of the intersection of $\pi$ with the horizontal line $y = k + 1/2$. We denote by $\pi(i,j)$ the type of the vertex at $(i,j) \in D_N$. We also let $f : \mathsf{Sign}_N^+ \rightarrow \mathbb{R}$ be given by 
$$f(\lambda; {\bf w}):= \mathsf{G}^c_{\mu}(w_1,\dots,w_J, w_{J+1}, \dots, w_{K}) = \sum_{ \mu \in \mathsf{Sign}_N^+} \mathsf{G}^c_{\mu}(w_1,\dots,w_J) \mathsf{G}^c_{\lambda/ \mu}(w_{J+1},\dots,w_{K}),$$
where the equality above follows from Proposition \ref{Branching}. With the above data, we define the weight of a collection of paths $\pi$ by
$$\mathcal{W}^f_{{\bf u}, {\bf w}}(\pi) = \prod_{i = 1}^N \prod_{j = 1}^\infty w_{u_i}(\pi(i,j)) \times f(\lambda^N(\pi); {\bf w}).$$
Using the branching relations, Proposition \ref{Branching}, and the Cauchy identity (\ref{CauchyS2}) we get
\begin{equation}\label{CID}
\sum_{\pi \in \mathcal{P}_N'}\mathcal{W}^f_{{\bf u}, {\bf w}}(\pi) = (q;q)_N \prod_{i = 1}^N \left( \frac{1}{1 - su_i} \prod_{j = 1}^K \frac{1 - qu_i w_j}{1 - u_iw_j}\right)=:Z^f({\bf u}; {\bf w}).
\end{equation}
In view of the admissability conditions satisfied by ${\bf u}$ and ${\bf w}$, the above sum is in fact absolutely convergent.

We next fix $s = q^{-1/2}$, set $w_i = q^{i-1} \epsilon$ for $i = 1,\dots,J$ and put $v_j = w_{j + J}$ for $j = 1,\dots, M$. Here $\epsilon > 0$ is chosen sufficiently small so that the admissability condition is maintained. One shows that with the above specialization of parameters $f$ becomes a function of $\lambda, {\bf v}, \epsilon$ and $q^J$ and we denote it by $f_\epsilon(\lambda;{\bf v}, q^J)$. Specifically, if $q^J = X$ one has 
\begin{equation}\label{feq-1}
\begin{split}
f_\epsilon(\lambda;{\bf v}, X) = \sum_{ \nu \in \mathsf{Sign}_N^+}  \frac{(q;q)_N(-q)^{n_0 - N}(s\epsilon;q)_{N - n_0}}{(q;q)_{n_0}(s\epsilon;q)_N(s\epsilon^{-1}; q^{-1})_{N - n_0}}
(Xq^{ -N + n_0 + 1};q)_{N-n_0} (s\epsilon X;q)_{n_0} \times \\
\prod_{i = 1}^{\infty}{\bf 1}_{\{ n_i \leq 1 \}}\prod_{j = 1}^{N-n_0}\left( \frac{1}{1 - s\epsilon q^{j-1}} \left( \frac{\epsilon q^{j-1} - s}{ 1 - s\epsilon q^{j-1}}\right)^{\nu_j} \right)\mathsf{G}^c_{\lambda/ \nu}(v_1,\dots,v_M),
\end{split}
\end{equation}
where $\nu = 0^{n_0} 1^{n_1}2^{n_2} \cdots$ and ${\bf 1}_E$ is the indicator function of a set $E$. In addition, specializing our $w$ variables in $Z^f({\bf u}; {\bf w})$ and replacing $q^J$ with $X$, we get that $Z^f({\bf u}; {\bf w})$ becomes
$$(q;q)_N \prod_{i = 1}^N \left( \frac{1}{1 - su_i}\frac{1 - X\epsilon u_i}{1 - \epsilon u_i} \prod_{j = 1}^M \frac{1 - qu_i v_j}{1 - u_i v_j} \right) =:Z^{f_\epsilon}({\bf u}; {\bf v},X) .$$

 We subsitute $X =  (s\epsilon)^{-1}$ and take the limit as $\epsilon \rightarrow 0$. Under this limit transition we have
\begin{equation}\label{fBoundary}
\begin{split}
&f(\lambda; {\bf v}, \rho):= \lim_{\epsilon \rightarrow 0} f_\epsilon(\lambda;{\bf v}, (s\epsilon)^{-1}) = \\
&  (-1)^N(q;q)_N  \hspace{-3mm} \sum_{ \nu \in \mathsf{Sign}_N^+} \hspace{-3mm} {\bf 1}_{\{n_0 = 0\}}\prod_{i = 1}^{\infty}{\bf 1}_{\{ n_i \leq 1 \}}\prod_{j = 1}^{N} \left(-s \right)^{\nu_j} \mathsf{G}^c_{\lambda/ \nu}(v_1,\dots,v_M),
\end{split}
\end{equation}
and 
$$Z^f({\bf u}) :=\lim_{\epsilon \rightarrow 0} Z^{f_\epsilon}({\bf u}; {\bf v},(s\epsilon)^{-1})  =  (q;q)_N \prod_{i = 1}^N \left( \frac{1 - s^{-1} u_i}{1 - su_i} \prod_{j = 1}^M \frac{1 - qu_i v_j}{1 - u_iv_j}\right).$$
The above formulas imply that $f(\lambda; {\bf v}, \rho) = 0$ if $\lambda_N = 0$ or $\lambda_i = \lambda_j$ for $i \neq j$.

Equation (\ref{CID}) can now be analytically extended in $X$ (both sides become polynomials in $X$), and after specializing $X =  (s\epsilon)^{-1}$ and taking the limit as $\epsilon \rightarrow 0^+$ we get
\begin{equation}\label{PFEqn3}
\sum_{\pi \in \mathcal{P}_N'}\prod_{i = 1}^N \prod_{j = 1}^\infty w_{u_i}(\pi(i,j)) \times f(\lambda^N(\pi); {\bf v}, \rho) = Z^f({\bf u}),
\end{equation}
where again the right side can be shown to be absolutely convergent. With $f(\lambda; {\bf v}, \rho)$ as above we define the following weight of a collection of paths in $\mathcal{P}'_N$
\begin{equation}\label{WS2}
\mathcal{W}^f_{{\bf u}, {\bf v}}(\pi) = \prod_{i = 1}^N \prod_{j = 1}^\infty w_{u_i}(\pi(i,j)) \times f(\lambda^N(\pi);{\bf v}, \rho).
\end{equation}
One can check that $\mathcal{W}^f_{{\bf u}, {\bf v}}(\pi)  \geq 0$, $\mathcal{W}^f_{{\bf u}, {\bf v}}(\pi) = 0$ for all $\pi \in \mathcal{P}'_N \setminus \mathcal{P}_N$. 
As weights are non-negative and the partition function $Z^f({\bf u})$ is positive and finite, we see from (\ref{PFEqn3}) that 
$$\mathbb{P}_{{\bf u}, {\bf v}}(\pi) := \frac{\mathcal{W}^f_{{\bf u}, {\bf v}}(\pi)}{Z^f({\bf u})},$$
defines an honest probability measure on $\mathcal{P}_N$. For future reference we summarize the parameter choices we have made in the following definition.

\begin{definition}\label{parameters}
Let $N,M \in \mathbb{N}$. We fix $q \in (0,1)$ and $s = q^{-1/2}$, ${\bf u} = (u_1,\dots,u_N)$ with $u_i > s$ and ${\bf v} = (v_1,\dots,v_M)$ with $v_j > 0$, and $\max_{i,j} u_iv_j < 1$. With these parameters, we denote $\mathbb{P}_{{\bf u}, {\bf v}}$ to be the probability measure on $\mathcal{P}_N$, defined above.
\end{definition}

We end this section with the following result that provides a formula for the finite-dimensional projections of $\mathbb{P}_{{\bf u}, {\bf v}}$.
\begin{lemma}\label{Projections}Let $N, M \in \mathbb{N}$. Fix $q \in (0,1)$ and $s = q^{-1/2}$, ${\bf u} = (u_1,\dots,u_N)$ with $u_i > s$ and ${\bf v} = (v_1,\dots,v_M)$ with $v_j > 0$, and $\max_{i,j} u_iv_j < 1$. With these parameters let $\mathbb{P}_{{\bf u}, {\bf v}}$ be as in Definition \ref{parameters}.  Let us fix $k \in \mathbb{N}$, $1\leq m_1 < m_2 < \cdots < m_k \leq N$ and $\mu^{m_i} \in \mathsf{Sign}_{m_i}^+$. Then
\begin{equation}\label{marg5}
\begin{split}
&\mathbb{P}_{{\bf u, v}} \left( \lambda^{m_1}(\pi) = \mu^{m_1},..., \lambda^{m_k}(\pi) = \mu^{m_k} \right) = \frac{\prod_{r = 0}^{k-1} \mathsf{F}_{\mu^{m_{r + 1}}/  \mu^{m_r}}(u_{m_r + 1},..., u_{m_{r + 1}})    f(\mu^{m_k};{\bf v}, \rho) }{Z^f({\bf u}, {\bf v}; m_k)},\\   &\mbox{ where } Z^f({\bf u}, {\bf v}; m_k) =  (q;q)_{m_k} \prod_{i = 1}^{m_k} \left( \frac{1 - s^{-1} u_i}{1 - su_i} \prod_{j = 1}^M \frac{1 - qu_i v_j}{1 - u_iv_j}\right).
\end{split}
\end{equation}
\end{lemma}
\begin{remark} If $k \leq N$ and $m_i = i$ for $i = 1,\dots, k$ then (\ref{marg5}) implies that the projection of $\mathbb{P}_{{\bf u}, {\bf v}}$ to $D_k$ has law $\mathbb{P}_{{\bf u}_k, {\bf v}}$, where ${\bf u}_k = (u_1, \dots, u_k)$. In particular, the measures $\mathbb{P}_{{\bf u}, {\bf v}}$ are consistent with each other and can be used to define a measure on up-right paths on the entire region $\mathbb{Z}^2_{\geq 0}.$
\end{remark}
\begin{proof} Equation (\ref{marg5}) can be found as \cite[Equation (29)]{dimitrov2016six} and we refer the interested reader to Section 2.3 in that paper for the proof.
\end{proof}

%
\section{Estimates for $f(\lambda; {\bf v},\rho)$ and $\mathsf{F}_{\lambda}$}\label{Section3} In this section we give a contour integral formula for the functions $f(\lambda; {\bf v},\rho)$, and a combinatorial estimate for the function $\mathsf{F}_{\lambda}$ from Definition \ref{DefF}. The results we derive in this section will be used in Sections \ref{Section4} and \ref{Section5} to prove Theorem \ref{thm:main}. 

%
\subsection{Integral formulas for $f(\lambda; {\bf v},\rho)$} The purpose of this section is to derive a contour integral formula for the function $f(\lambda; {\bf v},\rho)$ from Section \ref{Section2.2}. We accomplish this in Lemma \ref{CIFf} after we derive a contour integral formula for the functions $\mathsf{G}^c_{\lambda}$ from Definition \ref{def:G} in Lemma \ref{CIFg}. In the remainder of the paper we denote by $\iota$ the root $\sqrt{-1}$ that lies in the complex upper half-plane.
\begin{lemma}\label{CIFg} Suppose that $k,N \in \mathbb{N}$ satisfy $N \geq k$, $q \in (0,1)$, $s > 1$ and $v_1, \dots, v_N$ are complex numbers such that $|v_i| < s^{-1}$ for all $i =1, \dots ,N$. Then for any $\lambda \in \mathsf{Sign}_k^+$ with $\lambda_k \geq 1$ we have
\begin{equation}\label{CIFgEq}
\begin{split}
&\mathsf{G}_{\lambda}^c(v_1, \dots, v_N) = c(\lambda) (q;q)_k \cdot \oint_{\gamma} \cdots \oint_{\gamma}\prod_{1 \leq \alpha < \beta \leq k}  \frac{u_{\alpha}-u_{\beta}}{u_{\alpha}-q u_{\beta}} \times\\
& \prod_{i=1}^k\frac{1}{(1-s u_i)(u_i - s)} \left(\f{1-s u_i}{u_i-s}\right)^{\lambda_i} \prod_{i = 1}^k \prod_{j = 1}^N \frac{1-q u_i v_j}{1-u_i v_j} \prod_{i = 1}^k \frac{du_i}{2\pi \iota}.
\end{split}
\end{equation}
In (\ref{CIFgEq}) the constant $c(\lambda)$ is as in (\ref{relation}) and the contour $\gamma$ is a zero-centered positively oriented circle of radius $R \in \left(s, \min_{1, \dots, N} |v_i|^{-1} \right)$, where the latter set is non-empty by our assumption on $v_i$'s.
\end{lemma}
\begin{remark} We mention here that a similar result to the above lemma was proved in \cite[Corollary 7.16]{borodin2018higher} with several important differences. First of all, the formula in  \cite[Corollary 7.16]{borodin2018higher} is for the functions $\mathsf{G}_\lambda$ rather than $\mathsf{G}_{\lambda}^c$, but in view of (\ref{relation}) this difference is inessential. A more significant difference is that the authors in that paper assumed that $s \in (-1, 0)$ unlike our case of $s > 1$ -- this difference is also minor and can be overcome by an analytic continuation argument in the parameter $s$. A crucial difference is that the contour integral formula in \cite[Corollary 7.16]{borodin2018higher} is based on small contours that encircle $s$ while the contours in Lemma \ref{CIFg} above are large contours. In particular, the formulas in Lemma \ref{CIFg} are different and cannot be obtained by a direct application of Cauchy's theorem from the ones in \cite[Corollary 7.16]{borodin2018higher}. That being said, we mention that the existence of both small and large contour formulas is known in the related context of Macdonald processes, see \cite[Section 3.2.3]{borodin2014macdonald} and the derivation of both types of formulas is similar in spirit.
\end{remark}
\begin{proof} The proof is a standard computation of residues for the integrals on the right side of (\ref{CIFgEq}), but for clarity we split the proof into two steps. \\

{\bf \raggedleft Step 1.} We claim that (\ref{CIFgEq}) holds when $v_1, \dots, v_N \in (0, s^{-1})$ are such that $v_i \neq v_j$ for $i \neq j$.  We prove this statement in the second step. In the present step we assume its validity and conclude the proof of the lemma.\\

Let $\Omega$ denote the open disc of radius $s^{-1}$, centered at the origin in $\mathbb{C}$. Observe that by Definition \ref{def:G} and (\ref{eq:weightsconjugate}) we know that $\mathsf{G}_{\lambda}^c(v_1, \dots, v_N) $ is a finite sum of rational functions in $v_1, \dots, v_N$ that are analytic in $\Omega^N$ (here we used that the possible poles of $\mathsf{G}_{\lambda}^c$ come from the zeros of the denominators of $w_v(i_1, j_1; i_2, j_2)$ which are all located at $v = s^{-1}$). This means that for each $i = 1, \dots, N$ and $v_j \in \Omega$ for $j \neq i$ the left side of (\ref{CIFgEq}) as a function of $v_i$ is analytic in $\Omega$. Since $\gamma$ has radius bigger than $s$ by assumption, we see that for each $i = 1, \dots, N$ and $v_j \in \Omega$ for $j \neq i$  the integrand on the right side of (\ref{CIFgEq}) is also analytic on $\Omega$ as a function of $v_i$. It follows by \cite[Theorem 5.4]{stein2010complex} that for each $i = 1, \dots, N$ and $v_j \in \Omega$ for $j \neq i$ the right side of (\ref{CIFgEq}) is analytic on $\Omega$ as a function of $v_i$. 

We claim for each $k = 0, \dots, N$ that (\ref{CIFgEq}) holds if $v_1, \dots, v_{N-k} \in (0, s^{-1})$ are such that $v_i \neq v_j$ for $1 \leq i \neq j \leq N-k$ and $v_{N-k+1}, \dots v_N \in \Omega$. We prove this statement by induction on $k$ with base case $k= 0$ being true by our claim in the beginning of the step. Let us assume this result for $k$ and prove it for $k+1$. We fix $v_{N-k +1}, \dots, v_N \in \Omega$ and $v_1, \dots, v_{N- k - 1}$ in $(0,s^{-1})$ with $v_i \neq v_j$ for $1 \leq i \leq j \leq N-k - 1$. Put $m = \max(v_1, \dots, v_{N-k-1})$ and observe that by our discussion in the previous paragraph both sides of (\ref{CIFgEq}) are analytic functions of $v_{N-k}$ in $\Omega$ and by induction hypothesis these two functions are equal when $v_{N-k} \in (m, s^{-1})$. Since the latter set is contained in $\Omega$ and has a limit point in that set we conclude by \cite[Corollary 4.9]{stein2010complex} that both sides of (\ref{CIFgEq}) agree for all $v_{N-k} \in \Omega$. This proves the desired result for $k+1$ and we conclude by induction that the result holds when $k = N$, which is precisely the statement of the lemma.\\

{\bf \raggedleft Step 2.} In this step we prove that (\ref{CIFgEq}) holds when $v_1, \dots, v_N \in (0, s^{-1})$ are such that $v_i \neq v_j$ for $i \neq j$. We proceed to sequentially compute the integral with respect to $u_i$ for $i = 1,2, \dots, k$ in this order as a sum of residues. Observe that after we have evaluated the (minus) residues outside of $\gamma$ for $u_j$ with $j = 1, \dots, i-1$ the integrand only has simple poles when $u_i  = v_{m_i}^{-1}$ (there are no poles at infinity since the integrand is $\sim u_i^2$ as $|u_i| \rightarrow \infty$, and also no new poles are introduced after evaluating the residues at $u_j = v_{m_j}^{-1}$ for $j = 1, \dots, i-1$). Furthermore, the Vandermonde determinant $\prod_{1 \leq \alpha < \beta \leq k} (u_{\alpha}-u_{\beta})$ in the integrand implies that we only get a non-trivial contribution from the residues when $m_1, \dots, m_k$ are all distinct. Putting this all together, we conclude by the Residue theorem that the right side of (\ref{CIFgEq}) is equal to
\begin{equation*}
\begin{split}
&  c(\lambda) (q;q)_k  \sum_{I} \prod_{1 \leq \alpha < \beta \leq k}  \frac{v^{-1}_{m_\alpha}-v^{-1}_{m_\beta}}{v^{-1}_{m_\alpha}-q v^{-1}_{m_\beta}} \prod_{i=1}^k\frac{1}{(1-s v_{m_i}^{-1})(v_{m_i}^{-1} - s)} \left(\f{1-s v^{-1}_{m_i}}{v^{-1}_{m_i}-s}\right)^{\lambda_i} \times \\
&\frac{\prod_{i = 1}^k \prod_{j = 1}^N (v_j^{-1}-q v_{m_i}^{-1})}{\prod_{i = 1}^k \prod_{j = 1, j \neq m_i}^N(v_{j}^{-1}- v_{m_i}^{-1} )},
\end{split}
\end{equation*}
where the sum is over injective functions $I : \{1, \dots, k\} \rightarrow \{1, \dots, N\}$ and we have denoted $I(r) = m_r$. Performing some simplifications and rearrangements we conclude that 
\begin{equation}\label{A1}
\begin{split}
&  c(\lambda) (q;q)_k (1-q)^k \sum_{I} \prod_{1 \leq \alpha < \beta \leq k}  \frac{v_{m_\alpha}-q v_{m_\beta}}{v_{m_\alpha}-v_{m_\beta}} \prod_{i=1}^k\frac{v_{m_i}}{(v_{m_i}-s)(1- sv_{m_i})} \left(\f{v_{m_i}-s }{1-s v_{m_i}}\right)^{\lambda_i} \times \\
& \prod_{i = 1}^k \prod_{j \in J} \frac{(v_{m_i}-q v_{j})}{ (v_{m_i}- v_{j} )} = \mbox{RHS of (\ref{CIFgEq})},
\end{split}
\end{equation}
where $J = \{1, \dots, N\} \setminus \{m_1, \dots, m_k\}$. 

On the other hand, by (\ref{symmGEq}), we have that the left side of (\ref{CIFgEq}) is equal to
\begin{equation*}
\begin{split}
 &\frac{c(\lambda) (1-q)^N(q;q)_k    }{(q;q)_{N- k }}\sum_{\sigma \in S_N} 
 \sigma \left( \prod_{1 \leq \alpha < \beta \leq N} \frac{v_{\alpha} - q v_\beta}{v_\alpha - v_\beta} \prod_{i = 1}^k  \frac{v_i}{(v_i - s)(1- sv_i)} \left( \frac{ v_i - s}{1 - sv_i}\right)^{\lambda_i}\right),
\end{split}
\end{equation*}
where $\lambda = 0^{n_0} 1^{n_1}2^{n_2} \cdots.$ We next split the latter sum over the possible values of $\sigma(1), \dots, \sigma(k)$ and rewrite the above as
\begin{equation}\label{A2}
\begin{split}
 &\frac{c(\lambda) (1-q)^N(q;q)_k    }{(q;q)_{N- k }}\sum_{I} 
\prod_{1 \leq \alpha < \beta \leq k} \frac{v_{m_\alpha} - q v_{m_\beta}}{v_{m_\alpha} - v_{m_\beta}} \cdot \prod_{i = 1}^k \prod_{j \in J} \frac{(v_{m_i}-q v_{j})}{ (v_{m_i}- v_{j} )} \\
& \prod_{i = 1}^k \frac{v_{m_i}}{(v_{m_i}- s)(1- sv_{m_i})}  \left( \frac{ v_{m_i} - s}{1 - sv_{m_i}}\right)^{\lambda_i}\cdot \sum_{\tau \in S_{N-k}} \tau \left(\prod_{1 \leq \alpha < \beta \leq N-k} \frac{v_{j_{\tau(\alpha)}} - q v_{j_{\tau(\beta)}}}{v_{j_{\tau(\alpha)}} -  v_{j_{\tau(\beta)}}}  \right),
\end{split}
\end{equation}
where as before the sum is over injective maps $I: \{1, \dots, k\} \rightarrow \{1, \dots, N\}$, $m_r = I(r)$ for $r =1,\dots, k$ and $J = \{1, \dots, N \} \setminus \{m_1, \dots, m_k\}$. The inner sum is over pemutations $\tau$ of $\{1, \dots ,N-k\}$ and $j_1, \dots, j_{N-k}$ denote the elements of $J$ in increasing order (the particular order does not matter). We know from \cite[Chapter III, (1.4)]{Mac} that 
$$\sum_{\tau \in S_{N-k}} \tau \left(\prod_{1 \leq \alpha < \beta \leq N-k} \frac{v_{j_{\tau(\alpha)}} - q v_{j_{\tau(\beta)}}}{v_{j_{\tau(\alpha)}} -  v_{j_{\tau(\beta)}}}  \right) = \frac{(q;q)_{N-k}}{(1-q)^{N-k}}.$$
Substituting this in (\ref{A2}) we conclude that 
\begin{equation*}
\begin{split}
 &c(\lambda) (q;q)_k (1-q)^k  \sum_{I} 
\prod_{1 \leq \alpha < \beta \leq k} \frac{v_{m_\alpha} - q v_{m_\beta}}{v_{m_\alpha} - v_{m_\beta}} \cdot \prod_{i = 1}^k \prod_{j \in J} \frac{(v_{m_i}-q v_{j})}{ (v_{m_i}- v_{j} )} \\
& \prod_{i = 1}^k  \frac{v_{m_i}}{(v_{m_i}- s)(1- sv_{m_i})} \left( \frac{ v_{m_i} - s}{1 - sv_{m_i}}\right)^{\lambda_i} = \mbox{ LHS of (\ref{CIFgEq})}.
\end{split}
\end{equation*}
Comparing the last equation with (\ref{A1}) we conclude that the left and right sides of (\ref{CIFgEq}) agree when $v_1, \dots, v_N \in (0, s^{-1})$ are such that $v_i \neq v_j$ for $i \neq j$. This suffices for the proof.
\end{proof}

The next lemma provides a contour integral formula for $f(\lambda; {\bf v},\rho)$ from (\ref{fBoundary}).

\begin{lemma}\label{CIFf} Suppose that $k, M \in \mathbb{N}$, $q \in (0,1)$, $s = q^{-1/2}$ and $v_1, \dots, v_M \in (0, s^{-1})$. Then for any $\lambda \in \mathsf{Sign}_k^+$ with $\lambda_k \geq 1$ we have
\begin{equation}\label{CIFfEq}
\begin{split}
&f(\lambda; {\bf v},\rho)= c(\lambda) (q;q)_k \cdot \oint_{\gamma} \cdots \oint_{\gamma}\prod_{1 \leq \alpha < \beta \leq k}  \frac{u_{\alpha}-u_{\beta}}{u_{\alpha}-q u_{\beta}} \times\\
& \prod_{i=1}^k\frac{1}{-s(1-s u_i)} \left(\f{1-s u_i}{u_i-s}\right)^{\lambda_i} \prod_{i = 1}^k \prod_{j = 1}^M \frac{1-q u_i v_j}{1-u_i v_j} \prod_{i = 1}^k \frac{du_i}{2\pi \iota}.
\end{split}
\end{equation}
In (\ref{CIFfEq}) the constant $c(\lambda)$ is as in (\ref{relation}) and the contour $\gamma$ is a zero-centered positively oriented circle of radius $R \in \left(s, \min_{1, \dots, M} |v_i|^{-1} \right)$, where the latter set is non-empty by our assumption on $v_i$'s.
\end{lemma}
\begin{proof} We start from (\ref{CIFgEq}) with $N = M + J$, where $J \in \mathbb{N}$ and variables $w_1, \dots, w_N$ in place of $v_1, \dots, v_N$. We then set $w_{i} =  q^{i-1}\epsilon $ for $i = 1, \dots, J$ and $w_{J +i} = v_i$ for $i = 1, \dots, M$. Here $\epsilon \in (0, s^{-1})$. This gives
\begin{equation*}
\begin{split}
&\mathsf{G}_{\lambda}^c(\epsilon, q \epsilon, \cdots, q^{J-1} \epsilon, v_1, \dots, v_M ) = c(\lambda) (q;q)_k \cdot \oint_{\gamma} \cdots \oint_{\gamma}\prod_{1 \leq \alpha < \beta \leq k}  \frac{u_{\alpha}-u_{\beta}}{u_{\alpha}-q u_{\beta}} \times\\
& \prod_{i=1}^k\frac{1}{(1-s u_i)(u_i - s)} \left(\f{1-s u_i}{u_i-s}\right)^{\lambda_i} \prod_{i = 1}^k \prod_{j = 1}^M \frac{1-q u_i v_j}{1-u_i v_j} \cdot \prod_{i = 1}^k \frac{1-q^J \epsilon u_i }{1-u_i \epsilon }  \prod_{i = 1}^k \frac{du_i}{2\pi \iota}.
\end{split}
\end{equation*}
In particular, we see that if $f_\epsilon(\lambda;{\bf v}, X) $ is as in (\ref{feq-1}) we have
\begin{equation}\label{B1}
\begin{split}
&f_\epsilon(\lambda;{\bf v}, X) = c(\lambda) (q;q)_k \cdot \oint_{\gamma} \cdots \oint_{\gamma}\prod_{1 \leq \alpha < \beta \leq k}  \frac{u_{\alpha}-u_{\beta}}{u_{\alpha}-q u_{\beta}} \times\\
& \prod_{i=1}^k\frac{1}{(1-s u_i)(u_i - s)} \left(\f{1-s u_i}{u_i-s}\right)^{\lambda_i} \prod_{i = 1}^k \prod_{j = 1}^M \frac{1-q u_i v_j}{1-u_i v_j} \cdot \prod_{i = 1}^k \frac{1-X\epsilon u_i }{1-u_i \epsilon }  \prod_{i = 1}^k \frac{du_i}{2\pi \iota},
\end{split}
\end{equation}
whenever $X = q^J$ for any $J \geq 1$. In view of (\ref{feq-1}) we see that both sides of (\ref{B1}) are degree $k$ polynomials in $X$ and since they agree for infinitely many points $X = q^J, J \geq 1$ they must agree for all $X$. Then if we set $X = (s\epsilon)^{-1}$ and let $\epsilon \rightarrow 0$ we conclude in view of (\ref{fBoundary}) that 
\begin{equation*}
\begin{split}
&f(\lambda;{\bf v}, \rho) = \lim_{\epsilon \rightarrow 0} c(\lambda) (q;q)_k \cdot \oint_{\gamma} \cdots \oint_{\gamma}\prod_{1 \leq \alpha < \beta \leq k}  \frac{u_{\alpha}-u_{\beta}}{u_{\alpha}-q u_{\beta}} \times\\
& \prod_{i=1}^k\frac{1}{(1-s u_i)(u_i - s)} \left(\f{1-s u_i}{u_i-s}\right)^{\lambda_i} \prod_{i = 1}^k \prod_{j = 1}^M \frac{1-q u_i v_j}{1-u_i v_j} \cdot \prod_{i = 1}^k \frac{1-s^{-1} u_i }{1-u_i \epsilon }  \prod_{i = 1}^k \frac{du_i}{2\pi \iota},
\end{split}
\end{equation*}
which clearly implies (\ref{CIFfEq}) by the bounded convergence theorem.
\end{proof}

%
\subsection{Combinatorial estimates for $\mathsf{F}_\lambda$} We continue to use the notation from Section \ref{Section2}. In this section we estimate the function $\mathsf{F}_{\lambda}$ when $q \in (0,1)$, $s = q^{-1/2}$, $\lambda \in \mathsf{Sign}_k^+$ and $u_1, \dots, u_k$ are all equal to the same parameter $u > s$. We denote this function by $\mathsf{F}_{\lambda}([u]^k)$. For $\lambda \in \mathsf{Sign}_k^+$ we denote $|\lambda| = \lambda_1 + \cdots + \lambda_k$. The purpose of this section is to establish the following result.

\begin{lemma} \label{FBound} Fix $k \in \mathbb{N}$, $q \in (0,1)$, $s = q^{-1/2}$ and $u > s$. Then there exists a constant $C > 0$ depending on $k, q, u$ such that for all $\lambda \in \mathsf{Sign}_k^+$ with $\lambda_1 > \lambda_2 > \cdots > \lambda_k$ we have
\begin{equation}\label{FSqueeze}
\begin{split}
&\prod_{1 \leq i < j \leq k} \frac{\lambda_i-\lambda_j + j - i}{j-i}  - C \cdot (\lambda_1 -\lambda_k + k)^{\binom{k}{2} - 1} \leq \mathsf{F}_\lambda([u]^k)  \left( \frac{1-q}{1-su} \right)^{-\binom{k+1}{2}} \times \\
& \left( \frac{(1-q^{-1})u}{1-su}\right)^{-\binom{k}{2}}   \left( \frac{u-s}{1-s u} \right)^{-|\lambda|+\binom{k}{2}} \leq \prod_{1 \leq i < j \leq k} \frac{\lambda_i-\lambda_j+j - i}{j-i} + C \cdot(\lambda_1 -\lambda_k + k)^{\binom{k}{2} - 1}.
\end{split}
\end{equation}
\end{lemma}
We give the proof of Lemma \ref{FBound} in the end of the section. The general idea of the proof is as follows. From Definition \ref{DefF} the function $\mathsf{F}_{\lambda}([u]^k)$ is equal to a sum of weights $\mathcal{W}(\pi)$. For the majority of path collections $\pi$, which we call typical -- see Definition \ref{DefTypical} below, we have that the weight $\mathcal{W}(\pi)$ is equal to 
$$\mathcal{W}_{typ} = \left( \frac{1-q}{1-su} \right)^{\binom{k+1}{2}} \left( \frac{(1-q^{-1})u}{1-su}\right)^{\binom{k}{2}}  \cdot \left( \frac{u-s}{1-s u} \right)^{|\lambda|-\binom{k}{2}}.$$
We prove this in Lemma \ref{lem:typicalweight} below. We show that the weights $\mathcal{W}(\pi)$ for all path collections $\pi$ are within a constant multiple of the above weight -- we do this in Lemma \ref{lem:weightbound} below. Combining these two statements one deduces that $\mathsf{F}_\lambda([u]^k)  \approx \mathcal{W}_{typ} \times K$ where $K$ is the number of typical path collections. By a counting argument one can show that $K \approx \prod_{1 \leq i < j \leq k} \frac{\lambda_i-\lambda_j + j - i}{j-i}.$ Combining these three statements one obtains Lemma \ref{FBound}. We now turn to filling in the details of the above outline.

\begin{lemma} \label{lem:weightbound}
Fix $k \in \mathbb{N}$, $q \in (0,1)$, $s = q^{-1/2}$ and $u > s$. Let $\lambda \in \mathsf{Sign}_k^+$ and $\pi \in \mathcal{P}_{\lambda /\varnothing}$. Then there is a constant $\tilde{C}$ that depends on $k, q, u$ such that 
\begin{equation}  \label{eq:weightbound1}
\left| \mathcal{W}(\pi) \right| \leq \tilde{C} \left( \frac{u-s}{su - 1} \right)^{|\lambda|},
\end{equation}
where $|\lambda| = \lambda_1 + \cdots+ \lambda_k$ and $\mathcal{W}(\pi)$ is as in (\ref{defWeights}) for $u_1 = \cdots =u_k = u$. 
\end{lemma}
\begin{proof} From the definition of $\mathcal{P}_{\lambda/\varnothing}$ we know that a path collection $\pi$ has $|\lambda|$ horizontal edges and $\binom{k+1}{2}$ vertical edges  in $\mathbb{Z}_{\geq 0}^2$. Each edge borders two vertices except the top $k$ vertical edges whose upper vertex is not counted. Due to this there are at most $2 \binom{k+1}{2} - k=k^2$ vertices adjacent to a vertical edge. If we associate to each horizontal edge its left vertex and to each vertical edge its bottom vertex we obtain a surjective map from the set of edges to the set of vertices in $\pi$, whose path configuration is not $(0,0;0,0)$. Consequently, there are at most $|\lambda|+ \binom{k+1}{2}$ nontrivial (i.e. not type $(0,0;0,0)$) vertices in $\pi$. Also the above mapping from horizontal edges to their left vertices contains all vertices of type $(0,1;0,1)$ in its range and the pre-image of each such vertex contains exactly one element. This implies that the number of vertices of type $(0,1;0,1)$ is at least $|\lambda| - k^2$. 

We now recall from (\ref{eq:weights}) that 
$$w_u(0,0;0,0) = 1 \mbox{ and } w_u(0,1;0,1) = \frac{u-s}{1-su}.$$
Let $C \geq 1$ be a constant such that 
$$|w_u(i_1, j_1; i_2; j_2) | \leq C$$
for all $i_1, i_2 \in  \{0, \dots, k\}$ and $j_1, j_2 \in \{0,1\}$. The existence of $C$ is ensured by (\ref{eq:weights}) and it depends on $k,u,q$ (here $s =q^{-1/2}$). Then our work from the previous paragraph and (\ref{defWeights}) suggest that 
$$|\mathcal{W}(\pi)| \leq C^{\binom{k+1}{2} + k^2} \cdot \left( \frac{u-s}{su - 1} \right)^{|\lambda| - k^2},$$
which clearly implies (\ref{eq:weightbound1}). 
\end{proof}

\begin{definition}\label{DefTypical} Let $k \in \mathbb{N}$ and $\lambda \in \mathsf{Sign}_k^+$ be such that $\lambda_1 > \lambda_2 > \cdots > \lambda_k$. We say that a path collection $\pi \in \mathcal{P}_{\lambda / \varnothing}$ is a \emph{typical} path collection of $\mathcal{P}_{\lambda / \varnothing}$ if it only contains vertices of type $(0,0;0,0), (0,1;0,1), (0,1;1,0)$ and $(1, 0; 0, 1)$. We denote the set of all typical path collections by $\mathcal{P}^{typ}_{\lambda / \varnothing}.$ See Figure \ref{S3_1}.
\end{definition}
\begin{figure}[h]
\includegraphics[scale=0.6]{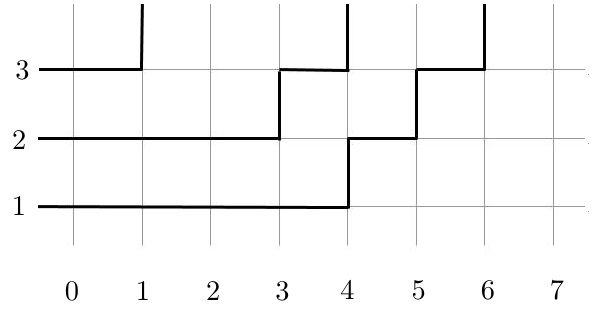}
 \captionsetup{width=0.9\linewidth}
\caption{Example of a path collection belonging to $\mathcal{P}^{typ}_{\lambda / \varnothing}$ where $\lambda = (6,3,1)$.
}
\label{S3_1}
\end{figure}

\begin{lemma} \label{lem:typicalweight} Fix $k \in \mathbb{N}$, $q \in (0,1)$, $s = q^{-1/2}$ and $u > s$.  Let $\lambda \in \mathsf{Sign}_k^+$ be such that $\lambda_1 > \lambda_2 > \cdots > \lambda_k$. If $\pi$ is a typical path collection of $\mathcal{P}_{\lambda / \varnothing}$, then 
\begin{equation}\label{TypicalWeight}
\mathcal{W}(\pi) = \left( \frac{1-q}{1-su} \right)^{\binom{k+1}{2}} \left( \frac{(1-q^{-1})u}{1-su}\right)^{\binom{k}{2}}  \cdot \left( \frac{u-s}{1-s u} \right)^{|\lambda|-\binom{k}{2}},
\end{equation}
where $|\lambda| = \lambda_1 + \cdots+ \lambda_k$ and $\mathcal{W}(\pi)$ is as in (\ref{defWeights}) for $u_1 = \cdots =u_k = u$. 
\end{lemma}

\begin{proof} Since $\pi$ is typical we know that it only contains vertices of type $(0,0;0,0)$, $(0,1;0,1)$, $(0,1;1,0)$ and $(1, 0; 0, 1)$. Furthermore, we have from (\ref{eq:weights}) that 
$$w_u(0,0;0,0) = 1,\hspace{2mm}  w_u(0,1;0,1) = \frac{u-s}{1-su},\hspace{2mm}  w_u(0,1;1,0) = \frac{1-q}{1-su}, \hspace{2mm} w_u(1,0;0,1) = \frac{(1-q^{-1})u}{1-su},$$
where we used that $s^2 = q^{-1}$. If $A$, $B$, $C$ denote the number of vertices in $\pi$ with path configuration $(0,1;0,1)$, $(0,1;1,0)$ and $(1,0;0,1)$ respectively then by (\ref{defWeights}) we know that 
$$\mathcal{W}(\pi) = \left( \frac{1-q}{1-su} \right)^{B} \left( \frac{(1-q^{-1})u}{1-su}\right)^{C}  \cdot \left( \frac{u-s}{1-s u} \right)^{A}.$$
Consequently, it suffices to show that if $\pi$ is typical then $A = |\lambda|-\binom{k}{2}$, $B = \binom{k+1}{2}$, $C = \binom{k}{2}$. 

We now proceed to simply count the the number of vertices of each type in a typical path collection. Notice that between row $i$ and row $i+1$ there are precisey $i$ vertical edges. The bottom vertex of each such edge has type $(0,1;1,0)$ and the top vertex of each such edge has type $(1,0;0,1)$. All other vertices in $\pi$ have type $(0,0;0,0)$ or $(0,1;0,1)$. We conclude from this that $C= 1 + 2 + \cdots + (k-1) = \binom{k}{2}$ and $B = 1+ 2 + \cdots + k = \binom{k+1}{2}$ (notice that the top vertex of the edges connecting row $k$ and $k+1$ are not included in the product defining $\mathcal{W}(\pi)$, while the bottom ones are). What we are left with is computing $A$. 

From the definition of $\mathcal{P}_{\lambda/\varnothing}$ we know that a path collection $\pi$ has $|\lambda|$ horizontal edges in $\mathbb{Z}_{\geq 0}^2$. The map that sends a horizontal edge to its left vertex endpoint maps the set of horizontal edges bijectively to the vertices of type $(0,1;0,1)$ and $(0,1;1,0)$ in $\pi$ and so $A + C =|\lambda|$. We conclude that $A = |\lambda|-\binom{k}{2}$ as desired.
\end{proof}

\begin{proof}(Lemma \ref{FBound}) Combining Lemmas \ref{lem:weightbound} and \ref{lem:typicalweight} we know that there is a constant $C_1$ that depends on $k, q, u$ such that 
\begin{equation}\label{C1}
\begin{split}
&|\mathcal{P}^{typ}_{\lambda / \varnothing}| - C_1\left( \left| \mathcal{P}_{\lambda / \varnothing} \right| -  \left| \mathcal{P}^{typ}_{\lambda / \varnothing}\right| \right)  \leq \mathsf{F}_\lambda([u]^k)  \left( \frac{1-q}{1-su} \right)^{-\binom{k+1}{2}} \times \\
& \left( \frac{(1-q^{-1})u}{1-su}\right)^{-\binom{k}{2}}   \left( \frac{u-s}{1-s u} \right)^{-|\lambda|+\binom{k}{2}} \leq  |\mathcal{P}^{typ}_{\lambda / \varnothing}|  + C_1\left( \left| \mathcal{P}_{\lambda / \varnothing} \right| -  \left| \mathcal{P}^{typ}_{\lambda / \varnothing}\right| \right) .
\end{split}
\end{equation}
From \cite[Equation (85)]{dimitrov2016six} we know that 
\begin{equation}\label{SizeAll}
\left| \mathcal{P}_{\lambda / \varnothing} \right|  =  \prod_{1 \leq i < j \leq k} \frac{\lambda_i-\lambda_j+j-i}{j-i}
\end{equation}
and from \cite[Equation (86)]{dimitrov2016six} we know that 
\begin{equation}\label{SizeTyp}
\left| \mathcal{P}^{typ}_{\lambda / \varnothing} \right|  \geq \prod_{1 \leq i < j \leq k} \frac{\lambda_i-\lambda_j -j + i}{j-i}.
\end{equation}
In particular, the equations (\ref{C1}), (\ref{SizeAll}) and (\ref{SizeTyp}) imply that 
\begin{equation*}
\begin{split}
&\prod_{1 \leq i < j \leq k} \frac{\lambda_i-\lambda_j + j - i}{j-i} - [C_1+ 1]\left( \left| \mathcal{P}_{\lambda / \varnothing} \right| -  \left| \mathcal{P}^{typ}_{\lambda / \varnothing}\right| \right)  \leq \mathsf{F}_\lambda([u]^k)  \left( \frac{1-q}{1-su} \right)^{-\binom{k+1}{2}} \times \\
& \left( \frac{(1-q^{-1})u}{1-su}\right)^{-\binom{k}{2}}   \left( \frac{u-s}{1-s u} \right)^{-|\lambda|+\binom{k}{2}} \leq  \prod_{1 \leq i < j \leq k} \frac{\lambda_i-\lambda_j+j-i}{j-i} + C_1\left( \left| \mathcal{P}_{\lambda / \varnothing} \right| -  \left| \mathcal{P}^{typ}_{\lambda / \varnothing}\right| \right) .
\end{split}
\end{equation*}
The latter equation now clearly implies (\ref{FSqueeze}) since 
$$0 \leq \left| \mathcal{P}_{\lambda / \varnothing} \right| -  \left| \mathcal{P}^{typ}_{\lambda / \varnothing}\right| \leq  \prod_{1 \leq i < j \leq k} \frac{\lambda_i-\lambda_j+j-i}{j-i} - \prod_{1 \leq i < j \leq k} \frac{\lambda_i-\lambda_j -j + i}{j-i} \leq C_2  (\lambda_1 -\lambda_k + k)^{\binom{k}{2} - 1},$$
for some sufficiently large constant $C_2 > 0$ depending on $k$ alone.
\end{proof}

%
\section{Proof of Theorem \ref{thm:main}}\label{Section4} In this section we prove Theorem \ref{thm:main}. We accomplish this in two steps. In the first step we prove that the random vectors $\left(Y_1^k(N,M;k), \dots, Y_k^k(N,M;k) \right)$ (i.e. the projections of the random vectors $Y(N,M;k)$ from Theorem \ref{thm:main} to their top row) weakly converge to the Hermite ensemble. In the second step we combine the convergence of $\left(Y_1^k(N,M;k), \dots, Y_k^k(N,M;k) \right)$ to the Hermite ensemble, with the fact that our model satisfies the six-vertex Gibbs property from \cite[Section 6]{dimitrov2016six} to conclude the convergence of $Y(N,M;k)$ to the GUE-corners process of rank $k$. 

%
\subsection{Convergence to the Hermite ensemble} We begin by recalling the joint distribution of the eigenvalues $\lambda_1 \leq \dots \leq \lambda_k$ of a $k\times k$ matrix from the GUE (recall that these were random Hermitian $k\times k$ matrices with density proportional to $e^{-Tr(X^2)/2}$). Specifically, from \cite[Equation (2.5.3)]{anderson2010introduction} we have the following formula.
\begin{definition} \label{def:GUEmeasure}
If $\mu_{\textrm{GUE}}^k$ denotes the joint distribution of the ordered eigenvalues of a random $k \times k$ GUE matrix, then $\mu_{\textrm{GUE}}^k$ has the following density with respect to Lebesgue measure
\begin{equation} \label{eq:GUEmeasuredef}
{\bf 1}\{ x_k > x_{k-1} > \cdots > x_1\} \left(\frac{1}{\sqrt{2 \pi}} \right)^k  \cdot \frac{1}{\prod_{i=1}^{k-1} i!} \cdot \prod_{1 \leq i < j \leq k} (x_i-x_j)^2 \prod_{i=1}^k e^{-\frac{x_i}{2}}.
\end{equation}
\end{definition}
\begin{remark} In the literature, the measure (\ref{eq:GUEmeasuredef}) is sometimes referred to as the {\em Hermite ensemble} due to its connection to Hermite orthogonal polynomials.
\end{remark}

The main result of this section is as follows.
\begin{proposition}\label{PropHerm} Under the same assumptions as in Theorem \ref{thm:main} we have that the random vectors $Y^k(N,M) = \left(Y_1^k(N,M;k), \dots, Y_k^k(N,M;k) \right)$ converge weakly to $\mu_{\textrm{GUE}}^k$ as $M \rightarrow \infty$. 
\end{proposition}

The starting point of our proof of Proposition \ref{PropHerm} is Lemma \ref{Projections}, from which we know that 
\begin{equation}\label{DefAB}
\begin{split}
&\mathbb{P}_{u,v}^{N,M}(\lambda^k(\pi) = \mu) = A_M(\mu) \cdot B_M(\mu), \mbox{ where } \\
&A_M(\mu)= \mathsf{F}_\mu([u]^k) \cdot  M^{-\binom{k}{2} \cdot (1/2)} \cdot  \left( \frac{1-q}{1-su} \right)^{-\binom{k+1}{2}} \left( \frac{(1-q^{-1})u}{1-su}\right)^{-\binom{k}{2}}  \cdot \left( \frac{u-s}{1-s u} \right)^{-|\mu|+\binom{k}{2}} \\
&B_M(\mu) =  f(\mu; [v]^M, \rho) \cdot M^{\binom{k}{2} \cdot (1/2)} \cdot  \left( \frac{1-q}{1-su} \right)^{\binom{k+1}{2}} \left( \frac{(1-q^{-1})u}{1-su}\right)^{\binom{k}{2}}  \cdot \left( \frac{u-s}{1-s u} \right)^{|\mu|- \binom{k}{2}} Z_M^{-1}, \\
& \mbox{ with  }Z_M = (q;q)_k \cdot \left(\frac{1-s^{-1} u}{1 - su} \right)^k \cdot \left( \frac{1-quv}{1-uv} \right)^{kM}.
\end{split}
\end{equation}
We recall that $\mathsf{F}_\mu([u]^k) $ stands for $\mathsf{F}_\mu$ with $u_1 = \cdots = u_k = u$ and also $f(\mu; [v]^M, \rho)$ stands for $f(\mu; {\bf v}, \rho)$ with $v_1 = \cdots = v_M = v$. We also recall that $|\mu| = \mu_1 + \cdots + \mu_k$.

The following lemma details the asymptotics of $A_M(\lambda)$ using the combinatorial estimates for $\mathsf{F}_\lambda([u]^k)$ from Lemma \ref{FBound}.
\begin{lemma}\label{AMAs} Suppose that $u,q,s$ satisfy $q \in (0,1)$, $s = q^{-1/2}$, $u > s$. Fix $a, A > 0$ and suppose that $x_1, \dots, x_k \in \mathbb{R}$ satisfy $A \geq x_k > x_{k-1} > \cdots > x_1 \geq -A$. Let $M_0(a,A) \geq 1$ be sufficiently large so that $aM_0 - A\sqrt{M_0} \geq 1$. For all $M \geq M_0$ we define $\lambda(M) \in \mathsf{Sign}^+_k$ through $\lambda_i(M) = \lfloor aM + \sqrt{M} x_{k-i+1} \rfloor$ for $i = 1, \dots, k$. Then we have 
\begin{equation}\label{ALimit}
\lim_{M \rightarrow \infty} A_M(\lambda(M)) = \prod_{1 \leq i < j \leq k} \frac{x_j- x_i}{j-i} =\frac{1}{\prod_{i=1}^{k-1} i!} \cdot  \prod_{1 \leq i < j \leq k} (x_j - x_i).
\end{equation} 
Moreoever, there is a constant $C > 0$ (it depends on $k, a, A, u ,q$) such that for all $M \geq M_0$ we have
\begin{equation}\label{ABound}
|A_M(\lambda(M))| \leq C.
\end{equation} 
\end{lemma}
\begin{proof}
We first prove (\ref{ABound}). From Lemma \ref{lem:weightbound} and Definition \ref{DefF} we know that 
$$|A_M(\lambda(M))| \leq \tilde{C} M^{-\binom{k}{2} \cdot (1/2)}  \cdot |\mathcal{P}_{\lambda(M) / \varnothing}| = \tilde{C}  \cdot\prod_{1 \leq i < j \leq k} \frac{\lambda_i(M)-\lambda_j(M)+j-i}{M^{1/2}(j-i)},$$
where in the last equality we used (\ref{SizeAll}) and $\tilde{C}$ is as in Lemma \ref{lem:weightbound} . Plugging in the definition of $\lambda_i(M)$ we see that for $M \geq M_0$ we have
$$|A_M(\lambda(M))| \leq \tilde{C} \prod_{1 \leq i < j \leq k} \frac{x_j -x_i+ 2k M^{-1/2}}{j-i} \leq \tilde{C} [2A + 2k ]^{\binom{k}{2}} ,$$
which clearly implies (\ref{ABound}). 

In the remainder of the proof we establish (\ref{ALimit}). By Lemma \ref{FBound} we know that there is a constant $C$ that depends on $k,u,q$ such that for all large enough $M$ we have
\begin{equation*}
\left|A_M(\lambda(M)) -   \prod_{1 \leq i < j \leq k} \hspace{-2mm}  \frac{\lambda_i(M)-\lambda_j(M)+j-i}{M^{1/2}(j-i)} \right| \leq C \cdot M^{- \binom{k}{2}\cdot (1/2)} \cdot[2AM^{1/2} + 1 + k]^{\binom{k}{2} - 1}.
\end{equation*}
Using that 
$$\lim_{M \rightarrow \infty} \prod_{1 \leq i < j \leq k} \left(\lambda_i(M)-\lambda_j(M)+j-i \right)M^{-1/2} =   \prod_{1 \leq i < j \leq k}(x_j- x_i),$$
we see that the above equation implies (\ref{ALimit}).
\end{proof}

The following lemma details the asymptotics of $B_M(\lambda)$. 
\begin{lemma}\label{BMAs} Suppose that $v,u,q,s,a,d$ are as in Definition \ref{DefConstants} and $k \in \mathbb{N}$. Fix $ A > 0$ and suppose that $x_1, \dots, x_k \in \mathbb{R}$ satisfy $A \geq x_k > x_{k-1} > \cdots > x_1 \geq -A$. Let $M_0(a,A) \geq 1$ be sufficiently sufficiently large so that $aM_0 - A\sqrt{M_0} \geq 1$. For all $M \geq M_0$ we define $\lambda(M) \in \mathsf{Sign}^+_k$ through $\lambda_i(M) = \lfloor aM + d\sqrt{M} x_{k-i+1} \rfloor$ for $i = 1, \dots, k$. Then we have 
\begin{equation}\label{BLimit}
\lim_{M \rightarrow \infty} d^{k} M^{k/2} B_M(\lambda(M)) =   d^{- \binom{k}{2}} \cdot (\sqrt{2\pi})^{-k}\prod_{1 \leq i < j \leq k} (x_j - x_i) \cdot \prod_{i = 1}^k e^{-x_i^2/2}.
\end{equation} 
Moreoever, there is a constant $C > 0$ (it depends on $k, a, A, u ,v,q$) such that for all $M \geq M_0$
\begin{equation}\label{BBound}
|d^{k} M^{k/2} B_M(\lambda(M))| \leq C.
\end{equation} 
\end{lemma}
Lemma \ref{BMAs} is the main technical result we need in the proof of Theorem \ref{thm:main}. The proof of this lemma is postponed until Section \ref{Section5}, and relies on a careful steepest descend analysis using the contour integral formula for $f(\mu; [v]^M, \rho)$ afforded by Lemma \ref{CIFf}.\\

In the remainder of this section we prove Proposition \ref{PropHerm}
\begin{proof} (Proposition \ref{PropHerm}) For clarity we split the proof into two steps.

{\bf \raggedleft Step 1.} Let $\mathbb{W}_k^{o}$ denote the open Weyl chamber in $\mathbb{R}^k$, i.e. 
$$\mathbb{W}^{o}_{k} := \{(x_1, \dots, x_k) \in \mathbb{R}^k: x_k > x_{k-1} > \cdots > x_1 \}.$$
Suppose that $R = [a_1, b_1] \times \cdots \times [a_k, b_k]$ is a closed rectangle such that $R \subset \mathbb{W}_k^{o}.$ The purpose of this step is to establish the following statement
\begin{equation}\label{limitRect}
\lim_{M \rightarrow \infty} \mathbb{P} \left( Y^k(N,M) \in R \right) = \int_R \mu_{GUE}^k(dx_1, \dots, dx_k).
\end{equation}

Let $A$ be sufficiently large so that $A \geq 1+ \max_{1 \leq i \leq k}|a_i| +\max_{1 \leq i \leq k}|b_i|.$ In addition if $M \in \mathbb{N}$ is given and $\mu \in \mathsf{Sign}_k^+$ we denote by $Q_\mu$ the cube
$$Q_\mu = \left[ \mu_k, \mu_k +1 \right) \times \cdots  \times \left[ \mu_1, \mu_1 + 1 \right).$$
We also write $L_i(M) =  \lceil a_i d\sqrt{M} +aM \rceil$ and $U_i(M) = \lfloor b_1 d\sqrt{M} + aM \rfloor$ for $i = 1,\dots, k$.

We first observe that for all sufficiently large $M$ we have
\begin{equation}\label{PrelimitRect}
\begin{split}
&\mathbb{P} \left( Y^k(N,M) \in R \right)  = \sum_{\lambda_1 = L_k(M)}^{U_k(M)} \cdots   \sum_{\lambda_k = L_1(M)}^{U_1(M)} \mathbb{P}_{u,v}^{N,M} \left( \lambda^k_i(\pi) = \lambda_i \mbox{ for $i = 1, \dots, k$} \right) = \\
&    \int_{[-A, A]^k} f_M(x_1, \dots, x_k) dx_1 \cdots dx_k,
\end{split}
\end{equation}
where $f_M(x)$ is a step function that is given by $d^{k} M^{k/2} A_M(\mu) B_M(\mu) $ if $x d \sqrt{M} + {\bf 1}_kaM \in Q_\mu \mbox{ for some } \mu = (\mu_1, \dots, \mu_k) \in \mathsf{Sign}^+_k \mbox{ such that } L_i(M) \leq \mu_{k-i+1} \leq U_i(M) \mbox{ for } i = 1, \dots, k; $
and $f_M(x) = 0$ otherwise. In the latter formula ${\bf 1}_k$ is the vector in $\mathbb{R}^k$ with all coordinates equal to $1$. 

By Lemmas \ref{AMAs} and \ref{BMAs} we know that for almost every $x \in [-A, A]^k$ we have
$$\lim_{M \rightarrow \infty} f_M(x_1, \dots, x_k) \rightarrow {\bf 1}_{R} \cdot \left(\frac{1}{\sqrt{2 \pi}} \right)^k \cdot  \frac{1}{\prod_{i=1}^{k-1} i!} \cdot \prod_{1 \leq i < j \leq k} (x_i-x_j)^2 \prod_{i=1}^k e^{-\frac{x_i}{2}}$$
 and $|f_M(x)| \leq C$ for some $C$ that depends on $A, u, q, v, k$ alone. Consequently, by the bounded convergence theorem we see that the $M \rightarrow \infty$ limit of (\ref{PrelimitRect}) implies (\ref{limitRect}).\\

{\bf \raggedleft Step 2.} The main goal of this step is to prove the following statement. For any open set $U$ with $U \subset \mathbb{W}_k^{o}$ we have that 
\begin{equation}\label{limitinfOpen}
\liminf_{M \rightarrow \infty} \mathbb{P} \left( Y^k(N,M) \in U  \right) \geq \int_U \mu_{GUE}^k(dx_1, \dots, dx_k).
\end{equation}
If we assume the validity of (\ref{limitinfOpen}) then we have that for any open set $O \subset \mathbb{R}^k$, 
$$\liminf_{M \rightarrow \infty} \mathbb{P} \left( Y^k(N,M) \in O \right) \geq\liminf_{M \rightarrow \infty} \mathbb{P} \left( Y^k(N,M) \in O \cap \mathbb{W}^o_k \right) \geq $$
$$\int_{O \cap \mathbb{W}_k^o} \mu_{GUE}^k(dx_1, \dots, dx_k) =\int_{O} \mu_{GUE}^k(dx_1, \dots, dx_k) ,$$
where in the last equality we used that the density of $\mu_{GUE}^k$ is zero outside of $\mathbb{W}_k^o$. The latter inequality and \cite[Theorem 3.2.11]{Durrett} imply the weak convergence of $Y_k(N,M)$ to $\mu_{\textrm{GUE}}^k$. Thus it suffices to prove (\ref{limitinfOpen}).

Let $U$ be an open subset of $ \mathbb{W}_k^{o}$. Then by \cite[Chapter 1, Theorem 1.4]{SS} we know that $U = \cup_{i = 1}^\infty R_i$ where $R_i$ are closed rectangles with disjoint interiors. Let $n \in \mathbb{N}$ and $\epsilon > 0$ be given. For $i = 1, \dots, n$ we let
$$R_i^{\epsilon} = [a^i_1 + \epsilon, b_1^i - \epsilon] \times \cdots \times [a^i_k + \epsilon, b_k^i - \epsilon] \mbox{ where } R_i = [a^i_1 , b_1^i ] \times \cdots \times [a^i_k , b_k^i ].$$
Using our result from Step 1 we know that 
$$\liminf_{M \rightarrow \infty} \mathbb{P} \left( Y^k(N,M) \in U \right)  \geq \liminf_{M \rightarrow \infty} \mathbb{P} \left( Y^k(N,M) \in \cup_{i = 1}^n R_i^{\epsilon} \right)  = $$
$$\liminf_{M \rightarrow \infty} \sum_{i = 1}^n  \mathbb{P} \left( Y^k(N,M) \in R_i^{\epsilon} \right) = \sum_{i = 1}^n  \int_{R_i^{\epsilon}} \mu_{GUE}^k(dx_1, \dots, dx_k).$$
Letting $\epsilon \rightarrow 0$ and applying the dominated convergence theorem with dominating function
$${\bf 1}\{ x_k > x_{k-1} > \cdots > x_1\} \left(\frac{1}{\sqrt{2 \pi}} \right)^k  \frac{1}{\prod_{i=1}^{k-1} i!} \cdot \prod_{1 \leq i < j \leq k} (x_i-x_j)^2 \prod_{i=1}^k e^{-\frac{x_i}{2}} $$
we conclude that 
$$\liminf_{M \rightarrow \infty} \mathbb{P} \left( Y^k(N,M) \in U \right)  \geq \sum_{i = 1}^n  \int_{R_i} \mu_{GUE}^k(dx_1, \dots, dx_k).$$
Letting $n \rightarrow \infty$ and using the monotone convergence theorem we conclude that (\ref{limitinfOpen}) holds.
\end{proof}

%
\subsection{Gibbs properties}\label{Section4.2} In this section we give the proof of Theorem \ref{thm:main}. The proof will be an easy consequence of Proposition \ref{PropHerm} and the fact that $\mathbb{P}^{N,M}_{u,v}$ satisfies what is known as the six-vertex Gibbs property, while the GUE-corners process satisfies what is known as the continuous Gibbs property. We start by explaining the latter two Gibbs properties. Our discussion will be brief, and we refer the interested reader to \cite[Sections 5 and 6]{dimitrov2016six} for a more detailed exposition.

We define several important concepts, adopting some of the notation from \cite{Gor14}. Let $\mathsf{GT}_k$ denote the set of $k$-tuples of {\em distinct} integers
$$\mathsf{GT}_{n} = \{ \lambda \in \mathbb{Z}^n: \lambda_1 < \lambda_2 < \cdots < \lambda_k\}.$$
We let $\mathsf{GT}_k^+$ be the subset of $\mathsf{GT}_k$ with $\lambda_1 \geq 0$.  We say that $\lambda \in \mathsf{GT}_k $ and $\mu \in \mathsf{GT}_{k-1}$ {\em interlace} and write $\mu \preceq \lambda$ if 
$$\lambda_1 \leq \mu_1 \leq \lambda_2 \leq \cdots \leq \mu_{k-1} \leq \lambda_k.$$

Let $\mathsf{GT}^{k}$ denote the set of sequences 
$$\mu^1 \preceq \mu^2 \preceq \cdots \preceq \mu^k, \hspace{3mm} \mu^i \in \mathsf{GT}_i, \hspace{2mm} 1 \leq i \leq k.$$
We call elements of $\mathsf{GT}^{k}$ {\em half-strict} Gelfand-Tsetlin patterns (they are also known as monotonous triangles, cf. \cite{MRR}). We also let $\mathsf{GT}^{k+}$ be the subset of $\mathsf{GT}^k$ with $\mu^k \in \mathsf{GT}_k^+$. For $\lambda \in \mathsf{GT}_k$ we let $\mathsf{GT}_\lambda \subset \mathsf{GT}^{k}$ denote the set of half-strict Gelfand-Tsetlin patterns $\mu^1 \preceq \cdots \preceq \mu^k$ such that $\mu^k = \lambda$.\\
 
We turn back to the notation from Section \ref{Section1.2} and consider $\pi \in \mathcal{P}_N$. For $k = 1,\dots,N$ we have that if we define $\mu^k_i(\pi) = \lambda^k_{k-i+1}(\pi)$ for $i = 1,\dots,k$ then $\mu^k \in \mathsf{GT}_k^+$. In addition, $\mu^{k+1} \succeq \mu^k$ for $k = 1,\dots,N-1$. Consequently, the sequence $\mu^1 ,\dots,\mu^k$ defines an element of $\mathsf{GT}^{k+}$. It is easy to see that the map $h : \mathcal{P}_k  \rightarrow \mathsf{GT}^{k+}$, given by $h(\pi) = \mu^1(\pi)  \preceq  \cdots \preceq \mu^k(\pi)$, is a bijection. For $\lambda \in \mathsf{GT}_k^+$ we let 
$$\mathcal{P}^{\lambda}_k = \{ \pi \in \mathcal{P}_k : \lambda^k_i(\pi) = \lambda_{k-i+1} \mbox{ for } i = 1,\dots,k\}.$$
One observes that by restriction, the map $h$ is a bijection between $\mathsf{GT}_\lambda$ and $\mathcal{P}^{\lambda}_k$. Given $\pi \in \mathcal{P}^{\lambda}_k$ and a vertex path configuration $(i_1, j_1; i_2, j_2)$ we let $N_{\pi, \lambda} (i_1,j_1; i_2, j_2)$ denote the number of vertices $(x,y) \in [1, \lambda_k] \times [1, k] \cap \mathbb{Z}^2$ with arrow configuration $(i_1,j_1; i_2, j_2)$. We abbreviate $N_1= N_{\pi,\lambda}(0, 0; 0, 0)$, $N_2 =N_{\pi, \lambda}(1, 1; 1, 1)$,  $N_3 = N_{\pi, \lambda}(1, 0; 1, 0)$, $N_4 = N_{\pi, \lambda}(0, 1; 0, 1)$, $N_5 = N_{\pi, \lambda}(1, 0; 0, 1)$, and $N_6 =N_{\pi, \lambda}(0, 1; 1, 0)$. 

With the above notation we make the following definition.
\begin{definition}\label{DGP}
Fix $w_1,w_2,w_3,w_4,w_5,w_6 > 0$. A probability distribution $\rho$ on $\mathsf{GT}^{k+}$ is said to satisfy the {\em six-vertex Gibbs property} (with weights $(w_1,w_2,w_3,w_4,w_5,w_6)$) if the following holds. For any $\lambda \in \mathsf{GT}_k^+$ such that 
$$\sum_{(\mu^1, \dots, \mu^k) \in \mathsf{GT}^{k+}: \mu^k = \lambda} \rho\left( \mu^1, \dots, \mu^k\right) > 0$$
 we have that the measure $\nu$ on $\mathcal{P}^\lambda_k$ defined through
$$\nu( h^{-1}(\omega)) = \rho(\omega| \mu^k = \lambda)$$
satisfies the condition
$$\nu(h^{-1}(\omega)) \propto w_1^{N_1}w_2^{N_2}w_3^{N_3}w_4^{N_4}w_5^{N_5}w_6^{N_6}.$$
In the above $ \rho(\cdot | \mu^k = \lambda)$ stands for the measure $\rho$ conditioned on $\mu^k= \lambda$ and the numbers $N_1, \dots, N_6$ are defined with respect to $\lambda$ and the path collection $\pi= h^{-1} (\omega)$. 
\end{definition}
\begin{remark} In simple terms, Definition \ref{DGP}, states that a probability measure on $\mathsf{GT}^{k+}$ satisfies the six-vertex Gibbs property if it can be realized from a measure of the type (\ref{GeneralMeasure}) with vertex weights $w_1, \dots, w_6$ for the six types of vertices under the bijection $h$. 
\end{remark}
One readily observes by the definition of $\mathbb{P}_{u,v}^{N,M}$ that if $\omega$ is $\mathbb{P}_{u,v}^{N,M}$-distributed and we define $\mu^j_i(\pi) = \lambda^j_{j-i+1}(\pi)$ for $1 \leq i \leq j \leq k$ then the law of $\left(\mu^j_i \right)_{1 \leq i \leq j \leq k}$ satisfies the six-vertex Gibbs property with weights
\begin{equation}\label{S6eqWT}
(w_1,w_2,w_3,w_4,w_5,w_6) =  \left( 1 , \frac{u - s^{-1}}{us - 1}, \frac{us^{-1} - 1}{us - 1}, \frac{u - s}{us - 1}, \frac{u(s^2-1)}{us - 1}, \frac{1 - s^{-2}}{us - 1} \right).
\end{equation}
The change of sign above compared to (\ref{eq:weightssix}) is made so that the above weights are positive (recall $u > s > 1$ in our case). \\

We next explain the continuous Gibbs property. We start by introducing some terminology from \cite{Def10} and \cite{Gor14}. Let $\mathcal{C}_n$ be the {\em Weyl chamber} in $\mathbb{R}^n$ i.e.
$$\mathcal{C}_n := \{ (x_1,...,x_n) \in \mathbb{R}^n: x_1 \leq  x_2 \leq  \cdots \leq  x_n \}.$$
For $x \in \mathbb{R}^n$ and $y \in \mathbb{R}^{n-1}$ we write $x \succeq y$ to mean that 
$$x_1 \leq y_1 \leq x_2 \leq y_2 \leq \cdots \leq x_{n-1} \leq y_{n-1} \leq x_n.$$
For $x = (x_1,...,x_n) \in \mathcal{C}_n$ we define the {\em Gelfand-Tsetlin polytope} to be
$$GT_n(x): = \{ (x^{1},...,x^{n}): x^{n} = x, x^{k} \in \mathbb{R}^k, x^{k} \succeq x^{k-1}, 2 \leq k \leq n\}.$$
We define the {\em Gelfand-Tsetlin cone} $GT^n$ to be
$$GT^n = \{ y \in \mathbb{R}^{n(n+1)/2}: y_{i }^{j+1} \leq y_{i}^j \leq y_{i+1}^{j+1}, \hspace{2mm} 1 \leq i \leq j \leq n-1\}.$$
We make the following definition after \cite{Gor14}.
\begin{definition}\label{definitionGT}
A probability measure $\mu$ on $GT^n$ is said to satisfy the {\em continuous Gibbs property} if conditioned on $y^n$ the distribution of $(y^1,...,y^{n-1})$ under $\mu$ is uniform on $GT_n(y^n)$. 
\end{definition}
\begin{remark} We refer the reader to \cite[Section 5]{dimitrov2016six} for a detailed discussion of the definition of the uniform measure on $GT_n(y)$, but in words the latter is a compact affine surface of finite dimension, which carries a natural uniform measure that is proportional to the Lebesgue measure on the affine space spanned by this surface.
\end{remark}

With the above notation we are finally ready to give the proof of Theorem \ref{thm:main}.
\begin{proof}(Theorem \ref{thm:main}) By Proposition \ref{PropHerm} we know that $Y^k(N,M) = \left(Y_1^k(N,M;k), \dots, Y_k^k(N,M;k) \right)$ converge weakly to $\mu_{\textrm{GUE}}^k$ as $M \rightarrow \infty$. Observe that by the interlacing conditions $\lambda^i(\pi) \preceq \lambda^{i+1}(\pi)$, for all $1 \leq i \leq k-1$ we have that 
$$ Y_1^k(N,M;k)\leq Y_i^j(N,M;k) \leq Y_k^k(N,M;k)\mbox{ for all $1\leq i \leq j \leq k$}. $$
Since $Y_1^k(N,M;k)$ and $Y_k^k(N,M;k)$ weakly converge we conclude from the last inequality that the random vectors $Y(N,M;k)$ are tight. 

Let $Y(\infty) =(Y_i^j(\infty): 1 \leq i \leq j \leq k)$ denote any subsequential limit of $Y(N(M),M;k)$, and let $Y(N(M_n),M_n;k)$ be a subsequence converging weakly to $Y(\infty) $. In view of Proposition \ref{PropHerm} we know that the joint distribution of $(Y_1^k(\infty), \dots, Y_k^k(\infty))$ is $\mu_{GUE}^k$. Furthermore, from our discussion earlier in the section, we know that the distribution of $\mu^j_i(\pi) = \lambda^j_{j-i+1}(\pi)$ for $1 \leq i \leq j \leq k$, where $\pi$ has distribution $\mathbb{P}_{u,v}^{N(M_n), M_n}$ satisfies the six-vertex Gibbs property with weights $w_1, \dots, w_6$ as in (\ref{S6eqWT}). We may now apply \cite[Proposition 6.7]{dimitrov2016six} and conclude that $Y(\infty)$ satisfies the continuous Gibbs property. We remark that in \cite[Proposition 6.7]{dimitrov2016six} the roles of $n$ and $k$ are swapped compared to our present notation and one should take $b(n) = d\sqrt{M_n}$ and $a(n) = a M_n$ in that proposition. 

Since $Y(\infty)$ satisfies the continuous Gibbs property and its top row $(Y_1^k(\infty), \dots, Y_k^k(\infty))$ has law $\mu_{GUE}^k$, we conclude that $Y(\infty)$ is the GUE-corners process of rank $k$. Since the sequence $Y(N,M;k)$ is tight and all weak subsequential limits are given by the GUE-corners process we conclude that $Y(N,M;k)$ converges weakly to the GUE-corners process  of rank $k$ as desired.
\end{proof}

%
\section{Asymptotic analysis}\label{Section5} In this section we prove Lemma \ref{BMAs}. We accomplish this in Section \ref{Section5.2} after we introduce some useful notation for the proof in Section \ref{Section5.1}.

%
\subsection{Setup}\label{Section5.1} Recall from Definition \ref{DefConstants} that our parameters $q,u,v,s$ satisfy 
\begin{equation} q \in (0,1), \qquad q=s^{-2}, \qquad 1<s<u<v^{-1}, \label{eq:parameterassumptions}
\end{equation}
 which we assume in what follows. If we assume the same notation as in Lemma \ref{BMAs} then in view of (\ref{DefAB}) and Lemma \ref{CIFf} we have for $M \geq M_0$ that
 \begin{equation}\label{BMNew}
\begin{split}
&d^{k} M^{k/2} B_M(\lambda(M)) =  d^{k}  M^{\binom{k + 1}{2} \cdot (1/2)} \cdot  \left( \frac{1-q}{1-su} \right)^{\binom{k+1}{2}} \left( \frac{(1-q^{-1})u}{1-su}\right)^{\binom{k}{2}}  \cdot \left( \frac{u-s}{1-s u} \right)^{- \binom{k}{2}} \cdot \\
&  \oint_{\gamma} \cdots \oint_{\gamma}\prod_{1 \leq \alpha < \beta \leq k}  \frac{u_{\alpha}-u_{\beta}}{u_{\alpha}-q u_{\beta}} \cdot  \prod_{i=1}^k\frac{s(1-su)}{(1-s u_i)(1-s^{-1}u)} \left(\f{1-s u_i}{u_i-s} \cdot \frac{u-s}{1-su} \right)^{\lambda_i(M)} \times\\
& \prod_{i = 1}^k \prod_{j = 1}^M \left( \frac{1-q u_i v_j}{1-u_i v_j} \cdot \frac{1- u v_j}{1-q u v_j} \right) \prod_{i = 1}^k \frac{du_i}{2\pi \iota}.
\end{split}
\end{equation}
Above, we can take $\gamma$ to be a zero-centered positively oriented circle of radius $u$ and we recall that $\iota = \sqrt{-1}$.

Recall from Definition \ref{DefConstants} the constants
\begin{equation}\label{constants}
\begin{split}
&a=\frac{ v \left(u-s^{-1}\right) \left(s^{-1} u-1\right)}{(1-u v) (1- s^{-2} u v)}, \qquad b=\frac{ (s^2-1)}{(u-s)(1-su)} \\
& c=\frac{1}{2} \left(a \left(\frac{1}{(u-s)^2}-\frac{s^2}{(1-s u)^2}\right)-\frac{s^{-4} v^2}{(1-s^{-2} u v)^2}+\frac{v^2}{(1-u v)^2} \right), \qquad d=\frac{- \sqrt{2 c}}{b}.
\end{split}
\end{equation}

We establish the following statement about the constants in (\ref{constants}).
\begin{lemma}\label{LemmaPar} For $u,v,s,q$ satisfying the conditions from \eqref{eq:parameterassumptions}, we have the following inequalities
$$a>0, \qquad b<0, \qquad c>0, \qquad d>0.$$
\end{lemma}
\begin{proof}
Since every factor in $a = \frac{ v \left(u-s^{-1}\right) \left(s^{-1} u-1\right)}{(1-u v) (1-q u v)}$ is positive we conclude that $a> 0$. Examining the factors of $b=\frac{ (s^2-1)}{(u-s)(1-su)}$ shows that $(1-su)$ is negative and the other factors are positive so $b < 0$. Once we show that $c$ is positive we will conclude that $d=-\frac{\sqrt{2 c}}{b}$ is also positive. Showing $c$ is positive requires a short argument that we present below.

Simplifying $c$ gives 
$$c=\frac{v(1-q)(1-s^{-1} v) T}{2 (s^{-1}-u) (s^{-1} u-1) (1-u v)^2 (1-q u v)^2},$$
where 
$$T=1+s^{-2}-2s^{-2} u v +s^{-3} u^2 v +s^{-1} u^2v-2 s^{-1} u.$$
From the above factorization formula for $c$, we see that to show that $c > 0$ it suffices to prove that $T< 0$. Let us put $v=y u^{-1}$ and $u=r s$ so that \eqref{eq:parameterassumptions} becomes the condition $r > 1$ and $0 <y<1$. In these variables we have
$$T(r,y)=1+q-2 q y +q y r+r y-2r = r( qy + y - 2) + (1 + q - 2qy).$$
The latter is a linear function in $r$ with a leading negative coefficient. Thus its maximum on $[1, \infty)$ is attained when $r = 1$ and then $T(1,y) = -(1-y)(1-q) < 0$. We conclude that $T(r,y) < 0$ for all $r > 1$ and $y \in (0,1)$, which proves that $c > 0$ as desired.
\end{proof}

\begin{definition}\label{DefAs} If $z \in \mathbb{C} \setminus \{0\}$ we define $\log(z) = \log|z| + \iota \phi$ where $z = |z| e^{\iota \phi}$ with $\phi \in (-\pi, \pi]$ (i.e. we take the principal branch of the logarithm). For $u_1, \dots ,u_k \in \mathbb{C}$ such that $u_i \neq q u_j$  and $u_i \neq s$ we define 
\begin{equation}\label{DefP}
p(u_1, \dots, u_k)= p(\vec{u}) = \prod_{1 \leq \alpha < \beta \leq k}  \frac{u_{\alpha}-u_{\beta}}{u_{\alpha}-q u_{\beta}} \cdot  \prod_{i=1}^k\frac{s(1-su)}{(1-s u_i)(1-s^{-1}u)}.
\end{equation}
We also define the functions 
\begin{equation}\label{DefGAs}
G(z) = a \cdot \log\left(\frac{1-sz}{z-s}\right) + \log \left(\frac{1-q z v}{1-z v}\right) -  a \cdot \log\left(\frac{1-su}{u-s}\right) + \log \left(\frac{1-q u v}{1-u v}\right),
\end{equation}
\begin{equation}\label{DefgAs}
g(z) =\log\left( \frac{1-s z}{z -s}\right) -\log \left(\frac{1-s u}{u-s}\right),
\end{equation}
and for $ x \in \mathbb{R}$ we let $h_M(x)$ be the unique element of $(-1,0]$ so that $aM+d x \sqrt{M}+h_M(x)$ is an integer. In the latter equations $q,u,v,s$ are as in (\ref{eq:parameterassumptions}) and $a,d$ are as in (\ref{constants}). Finally, we define 
\begin{equation}\label{FrontConstant}
A_k = d^{k} \cdot  \left( \frac{1-q}{1-su} \right)^{\binom{k+1}{2}} \left( \frac{(1-q^{-1})u}{1-su}\right)^{\binom{k}{2}}  \cdot \left( \frac{u-s}{1-s u} \right)^{- \binom{k}{2}}.
\end{equation}
\end{definition}

\begin{definition}\label{Contours} We let $\mathcal{C}$ denote the positively oriented contour that goes from $u - 2\iota u$ straight up to $u+ 2\iota u$ and then follows the half-cirlce of radius $2u$ centered at $u$, see Figure \ref{S5_1}. For $\epsilon \in (0,1)$ we also denote by $\mathcal{C}^{\epsilon}$ the contour that goes from $u- \iota \epsilon$ straight up to $u + \iota \epsilon$. 
\end{definition}
\begin{figure}[h]
\includegraphics[scale=0.8]{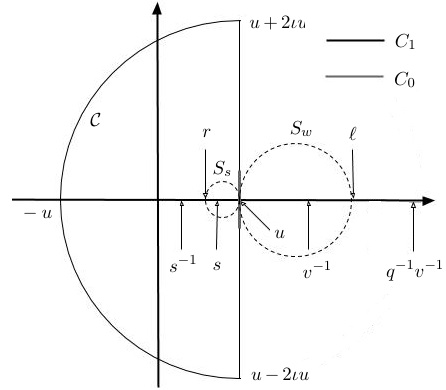}
 \captionsetup{width=0.9\linewidth}
\caption{The figure represents the contour $\mathcal{C}$ from Definition \ref{Contours}, in addition to $S_s, S_w$ as in the proof of Lemma \ref{techies} and the contours $C_0, C_1$ as in the proof of Lemma \ref{BMAs} in Section \ref{Section5.2} } 
\label{S5_1}
\end{figure}

 We may deform the $\gamma$ contours in (\ref{BMNew}) to the contour $\mathcal{C}$ from Definition \ref{Contours} without crossing any poles of the integrals, which by Cauchy's theorem does not change the value of the integral. After doing this contour deformation and utilizing the notation from Definition \ref{DefAs} we see that if $M \geq M_0$ we have
 \begin{equation}\label{BMNew2}
\begin{split}
&d^{-k} M^{k/2} B_M(\lambda(M)) = A_k \cdot   M^{\binom{k + 1}{2} \cdot (1/2)} \cdot \oint_{\mathcal{C}} \cdots \oint_{\mathcal{C}}p(\vec{u}) \cdot \\
&\exp\left( \sum_{i = 1}^k M G(u_i) + \sqrt{M} d x_i g(u_i) + h_M(x_i) g(u_i) \right) \prod_{i = 1}^k \frac{du_i}{2\pi \iota}.
\end{split}
\end{equation}

Our asymptotic analysis in the next section depends on a careful study of the functions $G$ and $g$ along the contour $\mathcal{C}$. We establish several useful properties in the following lemma.
\begin{lemma}\label{techies} Suppose that $G,g$ are as in Definition \ref{DefAs} and $\mathcal{C}, \mathcal{C}^{\epsilon}$ is as in Definition \ref{Contours}. We have 
\begin{equation}\label{M1}
G(u) = g(u) = G'(u) = 0, \hspace{5mm}  G''(u) =  2c , \hspace{5mm} g'(u) = b.
\end{equation}
For any $z \in \mathcal{C}$ we have that
\begin{equation}\label{M2}
Re [G(z)] \leq 0.
\end{equation} 
Moreover, for any $\epsilon \in (0,1)$ there exists $\delta > 0$ such that if $z \in \mathcal{C} \setminus \mathcal{C}^{\epsilon}$ 
\begin{equation}\label{M3}
Re [G(z)] \leq -\delta.
\end{equation} 
There exists $\epsilon_1 \in (0,1)$ and $C_1 > 0$ such that if $z \in \mathcal{C}^{\epsilon_1}$ we have that 
\begin{equation}\label{M4}
\left| G(z) -  c (z-u)^2 \right| \leq  C_1 |z-u|^3, \hspace{5mm}  2C_1 \epsilon_1 < c, \hspace{5mm} \left|g(z) -b(z-u)  \right| \leq C_1 |z-u|^2.
\end{equation} 
\end{lemma}
\begin{proof}
The fact that $G(u) = g(u) = 0$ is immediate from the definition. Next we have by a direct computation that 
$$G'(z) = a \cdot  \frac{q^{-1} -1}{(1 - sz)(z-s)} - \frac{v (1-q^{-1})}{(q^{-1} - vz) (1- vz)},$$
from which one checks directly (using the definition of $a$) that $G'(u) = 0$. Similar direct computations show that $G''(u) = 2c$ and $g'(u) = b$. 

By definition, we have that 
$$Re[G(z)] = a \log \left| \frac{z - s^{-1}}{z - s} \right| + \log \left| \frac{z - q^{-1} v^{-1}}{z - v^{-1}} \right| -  a \log \left| \frac{u - s^{-1}}{u - s} \right| - \log \left| \frac{u - q^{-1} v^{-1}}{u- v^{-1}} \right|.$$
Let $\ell$ denote the unique point in the segment $[s^{-1}, s]$ such that $\frac{\ell - s^{-1}}{s - \ell } = \frac{u - s^{-1}}{u - s}$, and $r$ be the unique point in the segment $[v^{-1} , q^{-v}v^{-1}]$ such that $\frac{q^{-1}v^{-1} - r}{r - v^{-1}} = \frac{q^{-1} v^{-1} - u}{v^{-1} - u}.$ We also denote by $S_s$ the circle, whose diameter is given by the segment $[\ell, u]$ and by $S_w$ the circle whose diameter is given by the segment $[u, r]$, see Figure \ref{S5_1}. 

The circles $S_s$ and $S_w$ are sometimes called Apollonius circles and they satisfy the properties 
$$\left| \frac{z - s^{-1}}{z- s} \right| \leq \frac{u - s^{-1}}{u - s} \mbox{ if $z$ lies outside of $S_s$ and }\left| \frac{z - s^{-1}}{z- s} \right| \geq \frac{u - s^{-1}}{u - s} \mbox{ if $z$ lies inside $S_s$};$$
 $$\left| \frac{z - q^{-1} v^{-1}}{z - v^{-1}} \right| \leq \frac{ q^{-1} v^{-1} - u}{v^{-1} - u}\mbox{ if $z$ lies outside of $S_w$ and } \left| \frac{z - q^{-1} v^{-1}}{z - v^{-1}} \right| \geq  \frac{q^{-1} v^{-1} - u}{v^{-1} - u},  $$
 if $z$ lies inside $S_w$. Since $\mathcal{C}$ lies outside of $S_w \cup S_s$ except for the point $u$ and $a > 0$  we conclude that for all $z \in \mathcal{C}$ we have $Re[G(z)] \leq Re[G(u)] = 0,$
while for any $z \in \mathcal{C}\setminus \{u\}$ we have $Re[G(u)] < Re[G(u)] = 0$. This proves (\ref{M2}) and by continuity of $G$ on $\mathcal{C}$ we also see that for any $\epsilon > 0$ there is a $\delta > 0$ such that (\ref{M3}) holds.
 
Finally, from our work above we know that in a neighborhood of $u$ we have
$$G(z) = c (z-u)^2 + O(|z-u|^3) \mbox{ and } g(z)= b (z- u) + O(|z-u|^2).$$
We can thus find $\epsilon_0 \in (0,1)$ and $C_1 > 0$ such that if $|z-u| \leq \epsilon_0$ we have
$$\left| G(z) -  c (z-u)^2 \right| \leq  C_1 |z-u|^3, \hspace{5mm} \left|g(z) - b(z-u) \right| \leq C_1 |z-u|^2.$$
Finally, since $c > 0$ we can pick $\epsilon_1 < \epsilon_0$ sufficiently small so that $2C_1 \epsilon_1 < c$ and then all the inequalities in (\ref{M4}) hold. This suffices for the proof.
\end{proof}

%
\subsection{The steepest descent argument}\label{Section5.2} In this section we prove Lemma \ref{BMAs}.

\begin{proof} (Lemma \ref{BMAs}) We follow the same notation as in Lemma \ref{BMAs} and Section \ref{Section5.1} above. For clarity we split the proof into four steps.\\

{\bf \raggedleft Step 1.} Let $\epsilon_1  \in (0,1)$ be as in the statement of Lemma \ref{techies}. We also let $\delta_1 > 0$ be as in Lemma \ref{techies} for $\epsilon = \epsilon_1$. We denote by $C_0$ the contour $\mathcal{C}^{\epsilon_1}$ and by $C_1$ the contour $\mathcal{C} \setminus \mathcal{C}^{\epsilon_1}$, see Figure \ref{S5_1}. We have that $\mathcal{C} = C_0 \cup C_1$ and $C_0$ is a small piece near $u$ while $C_1$ is the part of $\mathcal{C}$ away from $u$. In view of (\ref{BMNew2}) we have that if $M \geq M_0$ we have
 \begin{equation}\label{V1}
\begin{split}
&d^{-k} M^{k/2} B_M(\lambda(M)) = A_k \cdot   M^{\binom{k + 1}{2} \cdot (1/2)} \cdot  \sum_{ \sigma_1, \dots, \sigma_k \in \{0, 1\}}B(\sigma_1, \dots, \sigma_k), \mbox{ where }  \\
&B(\sigma_1, \dots, \sigma_k) = \oint_{C_{\sigma_1}} \hspace{-2mm} \cdots \oint_{C_{\sigma_k}} \hspace{-2mm} p(\vec{u})  \exp\left( \sum_{i = 1}^k M G(u_i) + \sqrt{M} d x_i g(u_i) + h_M(x_i) g(u_i) \right) \prod_{i = 1}^k \frac{du_i}{2\pi \iota}.
\end{split}
\end{equation}
In this step we prove that if $\sigma_1, \dots, \sigma_k \in \{0, 1 \}$ are such that $|\sigma| = \sigma_1 + \cdots + \sigma_k \geq 1$ we have that 
\begin{equation}\label{V2}
B(\sigma_1, \dots, \sigma_k) = O\left(e^{-(\delta_1/2)M} \right),
\end{equation}
where the constant in the big $O$ notation depends on $k, a, A, u ,v,q$.\\

Let $K_1, K_2 > 0$ be such that if $u_1, \dots, u_k, z \in \mathcal{C}$ we have
$$|g(z)| \leq K_1 \mbox{ and } |p(u_1, \dots, u_k)| \leq K_2.$$
Then in view of the definition of $\delta_1$, and equations (\ref{M2}), (\ref{M3}) we have that if $u_i \in C_{\sigma_i}$ for $i = 1, \dots, k$ we have
$$\left| p(\vec{u})  \exp\left( \sum_{i = 1}^k M G(u_i) + \sqrt{M} d x_i g(u_i) + h_M(x_i) g(u_i) \right)  \right| \leq  K_2 \exp \left( -M|\sigma|\delta_1 + \sqrt{M}k K_1[Ad + 1] \right) .$$
In deriving the above equation we used that $|e^{z}| \leq e^{|z|}$ for any complex $z$. The above equation now clearly implies (\ref{V2}). \\

{\bf \raggedleft Step 2.} In view of (\ref{V1}) and (\ref{V2}) we see that to prove the lemma it suffices to show that 
\begin{equation}\label{V3}
\lim_{M \rightarrow \infty} A_k \cdot   M^{\binom{k + 1}{2} \cdot (1/2)} B(0, \dots, 0) = d^{-\binom{k}{2}} \cdot (\sqrt{2\pi})^{-k}\prod_{1 \leq i < j \leq k} (x_j - x_i) \cdot \prod_{i = 1}^k e^{-x_i^2/2},
\end{equation}
and that there is a constant $C > 0$ depending on $k, a, A, u ,v,q$ such that 
\begin{equation}\label{V4}
| A_k \cdot   M^{\binom{k + 1}{2} \cdot (1/2)} B(0, \dots, 0) |  \leq C.
\end{equation}
In this step we prove (\ref{V4}). The proof of (\ref{V3}) is given in the next steps. \\

We perform a change of variables $u_i = u+ \iota \cdot M^{-1/2} \cdot y_i$ for $i = 1, \dots, k$. This gives the formula
\begin{equation}\label{V5}
\begin{split}
& A_k \cdot   M^{\binom{k + 1}{2} \cdot (1/2)} B(0, \dots, 0) =A_k \int_{\mathbb{R}^k}  \hat{p}_M(\vec{y})  \exp\left( \sum_{i =1}^k H_M(y_i)\right) \prod_{i = 1}^k {\bf 1} \{  |y_i| \leq \epsilon_1 M^{1/2} \} \frac{dy_i}{2\pi },
\end{split}
\end{equation}
where 
\begin{equation}\label{V6}
\begin{split}
&H_M(y) =  G(u + \iota \cdot M^{-1/2} y ) + \sqrt{M} d x_i g(u +\iota \cdot M^{-1/2} y ) + h_M(x_i) g(u + \iota \cdot M^{-1/2} y) \mbox{, and } \\
&  \hat{p}_M(\vec{y}) = \hspace{-3mm} \prod_{1 \leq \alpha < \beta \leq k}  \frac{\iota y_{\alpha}- \iota y_{\beta}}{(1 - q)u + \iota y_{\alpha}M^{-1/2} -q \iota y_{\beta}M^{-1/2}}  \prod_{i=1}^k\frac{s(1-su)}{(1-s u - s \iota M^{-1/2} y_i )(1-s^{-1}u)}.
\end{split}
\end{equation}
We see from (\ref{V6}) and (\ref{M4}) that there are constants $c_1, c_2 > 0$ that depend on $k, a, A, u ,v,q$ such that for all $M \in \mathbb{N}$ and $y_1, \dots, y_k \in \mathbb{R}$ we have
$$\left| \hat{p}_M(\vec{y}) \exp\left( \sum_{i =1}^k H_M(y_i)\right) \prod_{i = 1}^k {\bf 1} \{  |y_i| \leq \epsilon_1 M^{1/2} \}  \right| \leq h(\vec{y}),$$
where 
$$h(\vec{y}) = c_1 \cdot \prod_{1 \leq \alpha < \beta \leq k} | y_{\alpha}-  y_{\beta}| \cdot \exp\left( -(c/2)\sum_{i = 1}^k y_i^2 + c_2 \cdot \sum_{i = 1}^k |y_i| \right).$$
Combining the last inequality and (\ref{V5}) we conclude that for all $M \geq M_0$ we have 
$$ \left|A_k \cdot   M^{\binom{k + 1}{2} \cdot (1/2)} B(0, \dots, 0) \right| \leq A_k \int_{\mathbb{R}^k} h(\vec{y})  \prod_{i = 1}^k  \frac{dy_i}{2\pi },$$
which implies (\ref{V4}). \\

{\bf \raggedleft Step 3.} In this step we prove (\ref{V3}). From our work in the previous step we know that $h(\vec{y})$ is a dominating function for the functions
$$ \hat{p}_M(\vec{y}) \exp\left( \sum_{i =1}^k H_M(y_i)\right) \prod_{i = 1}^k {\bf 1} \{  |y_i| \leq \epsilon_1 M^{1/2} \},$$
which in view of (\ref{M4}) and (\ref{V6}) converge pointwise to
$$ \prod_{1 \leq \alpha < \beta \leq k}  \frac{\iota y_{\alpha}- \iota y_{\beta}}{(1 - q)u }  \prod_{i=1}^k\frac{s e^{-cy_i^2 + \iota d x_i y_i }}{1-s^{-1}u}.$$
Consequently, by the dominated convergence theorem, we conclude that 
\begin{equation}\label{V7}
\begin{split}
&\lim_{M \rightarrow \infty} A_k \cdot   M^{\binom{k + 1}{2} \cdot (1/2)} B(0, \dots, 0) = A_k \cdot \left((1-q)u \right)^{-\binom{k}{2}} \cdot \left( \frac{s}{1-s^{-1}u}\right)^k \times \\
& \int_{\mathbb{R}^k}  \prod_{1 \leq \alpha < \beta \leq k}  (\iota y_{\alpha}- \iota y_{\beta})  \prod_{i = 1}^k e^{-cy_i^2 + \iota d b x_i y_i }\frac{dy_i}{2\pi }.
\end{split}
\end{equation}

Substituting $A_k$ from (\ref{FrontConstant}) and performing the change of variables $z_i = \sqrt{2c} y_i$ (recall that $d = \frac{-\sqrt{2c}}{b}$) we obtain 
\begin{equation*}
\begin{split}
&\lim_{M \rightarrow \infty} A_k \cdot   M^{\binom{k + 1}{2} \cdot (1/2)} B(0, \dots, 0) =   d^{-\binom{k}{2}}   \int_{\mathbb{R}^k}  \prod_{1 \leq \alpha < \beta \leq k}  (\iota z_{\alpha}- \iota z_{\beta})  \prod_{i = 1}^k e^{-z_i^2/2 - \iota  x_i z_i }\frac{dz_i}{2\pi }.
\end{split}
\end{equation*}
We next use the formula for the Vandermonde determinant 
$$ \prod_{1 \leq \alpha < \beta \leq k}  (\iota z_{\alpha}- \iota z_{\beta})  = (\iota)^{\binom{k}{2}} \det \left[ z_i^{k- j}\right]_{i, j = 1}^k ,$$
and the linearity of the determinant to conclude that 
\begin{equation}\label{V8}
\begin{split}
&\lim_{M \rightarrow \infty} A_k \cdot   M^{\binom{k + 1}{2} \cdot (1/2)} B(0, \dots, 0) =   d^{-\binom{k}{2}} (\iota)^{\binom{k}{2}} \det \left[ \psi_{k-j}(x_i)\right]_{i, j = 1}^k, \mbox{ where } 
\end{split}
\end{equation}
$$\psi_{k-j}(x) = \int_{\mathbb{R}}   z^{k-j} e^{-z^2/2 - \iota  x z }\frac{dz}{2\pi }.$$

We claim that 
\begin{equation}\label{V9}
 (\iota)^{\binom{k}{2}} \det \left[ \psi_{k-j}(x_i)\right]_{i, j = 1}^k = (\sqrt{2\pi})^{-k}\prod_{1 \leq i < j \leq k} (x_j - x_i) \cdot \prod_{i = 1}^k e^{-x_i^2/2}.
\end{equation}
Notice that (\ref{V8}) and (\ref{V9}) together imply (\ref{V3}). We have thus reduced the proof of the lemma to establishing (\ref{V9}), which we do in the next and final step.\\

{\bf \raggedleft Step 4.} In this step we prove (\ref{V9}). Let $h_n(x)$ stands for the $n$-th Hermite polynomial, i.e.
\begin{equation}\label{V10}
h_n(x) = (-1)^n e^{\frac{x^2}{2}}\partial_x^n e^{-\frac{x^2}{2}},
\end{equation}
 see e.g. \cite[Section 3.2.1]{anderson2010introduction} for the definition and basic properties of these polynomials. Our first observation is that for $n \in \mathbb{Z}_{\geq 0}$ we have
\begin{equation}\label{V11}
  \psi_{n}(x) = (-\iota)^{n} (\sqrt{2 \pi })^{-1} e^{-\frac{x^2}{2}} h_n(x).
\end{equation}
We argue this by induction on $n$ with base case $n = 0$ being true in view of 
$$\psi_0(x) = \int_{\mathbb{R}}    e^{-z^2/2 - \iota  x z }\frac{dz}{2\pi } = (\sqrt{2\pi})^{-1} \cdot e^{-x^2/2},$$
where we used the formula for the characteristic function of a standard normal random variable. Suppose we know that (\ref{V11}) holds for $n$ and differentiate both sides with respect to $x$. For the right side we have using (\ref{V10}) that 
$$ \partial_x \left( (-\iota)^{n} (\sqrt{2 \pi })^{-1} e^{-\frac{x^2}{2}} h_n(x) \right) =  -(- \iota)^{-n} (\sqrt{2 \pi })^{-1} e^{-\frac{x^2}{2}} h_{n+1}(x) =  (- \iota)^{n + 2} (\sqrt{2 \pi })^{-1} e^{-\frac{x^2}{2}} h_{n+1}(x),$$
while for the left side we have
$$\partial_x \psi_n(x) =  \int_{\mathbb{R}}   z^{n} \partial_x e^{-z^2/2 - \iota  x z }\frac{dz}{2\pi } = (-\iota)  \int_{\mathbb{R}}   z^{n+1}  e^{-z^2/2 - \iota  x z }\frac{dz}{2\pi } = (-\iota) \psi_{n+1}(x),$$
where we can differentiate under the integral by the rapid decay of the integrand near infinity. The last two equations imply (\ref{V11}) for $n+1$ and so we conclude that (\ref{V11}) holds for all $n \in \mathbb{Z}_{\geq 0}$ by induction.

In view of (\ref{V11}) and the linearity of the determinant we see that to prove (\ref{V9}) it suffices to show that 
$$\det \left[ h_{k-j}(x_i)\right]_{i, j = 1}^k =\prod_{1 \leq i < j \leq k} (x_i - x_j) .$$
The latter is now clear since $h_n(x)$ is a monic polynomial of degree $n$, cf. \cite[(3.2.3)]{anderson2010introduction}, and so
$$\det \left[ h_{k-j}(x_i)\right]_{i, j = 1}^k = \det \left[ x_i^{k-j}\right]_{i, j = 1}^k = \prod_{1 \leq i < j \leq k} (x_i - x_j),$$
by the Vandermonde determinant formula. This suffices for the proof.
\end{proof}

\bibliographystyle{amsplain} 
\bibliography{PD}

\end{document}